\documentclass{article}
\usepackage{
amsmath,
amsfonts,
latexsym, 
amsrefs,
amssymb,
tikz-cd, 
bbm,
cases
}

\usepackage{bm}

\usepackage{tocloft}

\usepackage{cite}

\usepackage{graphicx}
\usepackage{enumerate}
\usepackage{mathrsfs}
\usepackage{wrapfig}

\usepackage{color}

\numberwithin{equation}{section}

\newcommand{\rd}{{\rm d}}

\newcommand{\be}{\begin{equation}}
\newcommand{\ee}{\end{equation}}

\newcommand{\e}{{\varepsilon}}

\usepackage{amsmath} %[intlimits]
\usepackage{amssymb}
\usepackage{amsthm}

\newcommand{\wt}{\widetilde}

\newcommand{\ii}{\mathrm{i}} %\newcommand{\mi}{\mathrm{i}}

\renewcommand{\epsilon}{\varepsilon}
\renewcommand{\leq}{\leqslant}
\renewcommand{\geq}{\geqslant}

\renewcommand{\P}{\mathbb{P}}
\newcommand{\E}{\mathbb{E}}

\DeclareMathOperator{\var}{Var}

\DeclareMathOperator{\im}{Im}

\DeclareMathOperator{\OO}{O}
\DeclareMathOperator{\oo}{o}

\theoremstyle{plain} %plain, definition, remark
\newtheorem{theorem}{Theorem}[section]
\newtheorem*{theorem*}{Theorem}
\newtheorem{lemma}[theorem]{Lemma}
\newtheorem*{lemma*}{Lemma}
\newtheorem{corollary}[theorem]{Corollary}
\newtheorem*{corollary*}{Corollary}
\newtheorem{proposition}[theorem]{Proposition}
\newtheorem*{proposition*}{Proposition}

\newtheorem{definition}[theorem]{Definition}
\newtheorem*{definition*}{Definition}

\newtheorem*{example*}{Example}
\newtheorem{remark}[theorem]{Remark}

\newtheorem*{remark*}{Remark}
\newtheorem*{remarks*}{Remarks}

\makeatletter
\renewcommand{\subsection}{\@startsection
{subsection}%                   % the name
{2}%                         % the level
{0mm}%                       % the indent
{-\baselineskip}%            % the before skip
{0 \baselineskip}%          % the after skip
{\normalfont\itshape}} % the style
\makeatother

\usepackage{dsfont}
\usepackage{stmaryrd}

\newcommand{\nc}{\normalcolor}

\def\@empty{}

\def\author#1{\par
    {\centering{\authorfont#1}\par\vspace*{0.05in}}
}

\def\titlefont{\fontsize{13}{15}\bfseries\boldmath\selectfont\centering{}}
\def\authorfont{\fontsize{13}{15}}
\def\abstractfont{\fontsize{8}{10}}

\let\affiliationfont\rhfont

\def\address#1{\par
    {\centering{\affiliationfont#1\par}}\par\vspace*{11pt}
}

\def\body{
\setcounter{footnote}{0}
\def\thefootnote{\alph{footnote}}
\def\@makefnmark{{$^{\rm \@thefnmark}$}}
}

\def\title#1{
    \thispagestyle{plain}
    \vspace*{-14pt}
    \vskip 79pt
    {\centering{\titlefont #1\par}}%
    \vskip 1em
}

\renewenvironment{abstract}{\par%
    \vspace*{6pt}\noindent %{\bf Abstract}
    \abstractfont
    \noindent\leftskip18pt\rightskip18pt
}{%
  \par}

\usepackage{tocloft}

\makeatletter
\renewcommand{\section}{\@startsection
{section}%                   % the name
{1}%                         % the level
{0mm}%                       % the indent
{-2\baselineskip}%            % the before skip
{1\baselineskip}%          % the after skip
{\normalfont\large\scshape\centering}} % the style
\makeatother

\begin{document}

~\vspace{-2cm}

\title{Gaussian fluctuations of the determinant of Wigner Matrices}

\vspace{1.2cm}
\noindent
\begin{minipage}[b]{0.5\textwidth}

 \author{P. Bourgade}

\address{New York University\\
   E-mail:  bourgade@cims.nyu.edu}

 \end{minipage}
\begin{minipage}[b]{0.5\textwidth}

 \author{K. Mody}

\address{New York University\\
   E-mail: km2718@nyu.edu}
 \end{minipage}

\begin{abstract}
We prove that the logarithm of the determinant of a Wigner matrix satisfies a central limit theorem in the limit of large dimension.
Previous results about fluctuations of such determinants required that the first four moments of the matrix entries match 
those of a Gaussian \cite{TaoVu2012}.
Our work treats symmetric and Hermitian matrices with centered entries having the same variance and subgaussian tail. 
In particular, it applies to symmetric Bernoulli matrices and answers an open problem raised in \cite{TaoVu2014}.
The method relies on 
(1) the observable introduced in \cite{BourgadeExtreme} and the stochastic advection equation it satisfies,
(2) strong estimates on the Green function as in \cite{CacMalSch2015},
(3) fixed energy universality \cite{BouErdYauYin2016},
(4) a moment matching argument \cite{TaoVu2011} using Green's function comparison \cite{ErdYauYin2012Bulk}.
\end{abstract}

\tableofcontents

\let\thefootnote\relax\footnote{\noindent The work of P.B. is partially supported by the NSF grant DMS\#1513587 and a Poincar\'e Chair.}

\section{Introduction}
	
	\noindent In this paper, we address the universality of the determinant of a class of random Hermitian matrices.
	Before discussing results specific to this symmetry assumption, we give a brief history of results in the non-Hermitian setting. In both settings, a priori bounds preceded estimates on moments of determinants, and the distribution of
	determinants for integrable models of random matrices. The universality of such determinants has then been the 
	subject of recent active research.

	\subsection{Non-Hermitian matrices.}\ 
	Early papers on this topic treat non-Hermitian matrices with independent and identically distributed entries.
		More specifically, Szekeres and Tur\'an first studied an extremal problem on  the determinant of $\pm 1$  matrices \cite{SzeTur1937}. 
	In the 1950s, a series of papers \cite{For1951,ForTuk1952, NyqRicRio1954, Tur1955, Pre1967} calculated 
	moments of the determinant of random matrices of fixed size (see also \cite{Gir1980}). 
	In general, explicit formulae are unavailable for high order moments of the determinant except
	when the entries of the matrix have particular distribution (see, for example, \cite{Dem1989} and the references therein).
	Estimates for the moments and the Chebyshev inequality give upper bounds
	on the magnitude of the determinant. \\
	
	\noindent Along a different line of research,
	for an $N\times N$ non-Hermitian random matrix $A$,
	Erd\H{o}s asked whether $\det A$ is non-zero with probability
	tending to one as $N$ tends to infinity.  In \cite{Kom1967, Kom1968}, Kolm\'{o}s proved
	that for random matrices with Bernoulli entries,
	indeed $ \det A \neq 0$ with probability converging to 1 with $N$. 
	In fact, this method works for more general models, and following \cite{Kom1967},
	\cite{KahKomSze1995, TaoVu2006, TaoVu2007, BouVuWoo2010} give improved, exponentially small bounds on the probability that $\det A = 0$. \\

	\noindent 
	In \cite{TaoVu2006}, the authors made the first steps towards quantifying the typical size of $\left| \det A \right|$,
	proving that for Bernoulli random matrices, with probability
	tending to 1 as $N$ tends to infinity,
		\begin{equation}
		\label{tao eq 1}
		\sqrt{N!}\exp\left(-c \sqrt{N\log N}\right) \leq \left| \det A\right| \leq \omega(N)\sqrt{N!}, 
		\end{equation}
	for any function $\omega(N)$ tending to infinity with $N$. In particular, with overwhelming probability
		\[ \log \left| \det A \right| = \left( \frac{1}{2} + \oo(1) \right) N\log N. \]
	
	\noindent In \cite{Goo1963}, Goodman considered $A$ with independent standard real Gaussian entries.
	In this case, he was able to express $\left| \det A \right|^2$ as the product of independent
	chi-square variables. This enables one to identify the asymptotic distribution
	of $\log\left| \det A \right|$. Indeed, one can prove that
		\begin{equation}
		\label{GOE result}
		 \frac{\log \left| \det A \right|  - \frac{1}{2}\log N! + \frac{1}{2}\log N}{\sqrt{\frac{1}{2}\log N}} 
			\to \mathscr{N}(0, 1),
		\end{equation}
	(see \cite{RemWes2005}). In the case of $A$ with independent complex Gaussian entries, a similar analysis yields 
		\[ \frac{\log \left| \det A \right|  - \frac{1}{2}\log N! + \frac{1}{4}\log N}{\sqrt{\frac{1}{4}\log N}} 
			\to \mathscr{N}(0, 1). \]
	
	\noindent 
	In \cite{NguVu2014}, the authors proved (\ref{GOE result}) holds under just
	an exponential decay hypothesis on the entries. Their method yields an explicit  rate of convergence
	and extends to handle the complex case. Then in \cite{BaoPanZho2015}, the authors extended 
	(\ref{GOE result}) to the case where the matrix entries only require bounded fourth moment.\\
	
	\noindent The analysis of determinants of non-Hermitian random matrices relies crucially on the
	assumption that the rows of the random matrix are independent. The fact that this independence
	no longer holds for Hermitian random matrices forces one to look for new methods to prove similar
	results to those of the non-Hermitian case. 
	Nevertheless, the history of this problem mirrors the history of the non-Hermitian case. 
	
	\subsection{Hermitian matrices.}\ In the 1980s, Weiss posed the Hermitian analogs of \cite{Kom1967, Kom1968} as an open problem.
	This problem was solved, many years later in \cite{CosTaoVu2006}, and 
	then in \cite[Theorem 34]{TaoVu2011} the authors proved the Hermitian analog of 
	(\ref{tao eq 1}). This left open the question of describing the limiting
	distribution of the determinant. \\	
	
	\noindent In \cite{DelLeC2000}, Delannay and Le Ca\"{e}r used
	the explicit formula for the joint distribution of the eigenvalues to prove
	that for $H$ an $N \times N$ matrix drawn from the GUE,
		\begin{equation}
		\label{GUE CLT}
		 \frac{\log \left| \det H \right|  - \frac{1}{2}\log N! + \frac{1}{4}\log N}{\sqrt{\frac{1}{2}\log N}} 
			\to \mathscr{N}(0, 1). 
		\end{equation}
	Analogously, one has
		\begin{equation}
		\label{GOE CLT}
		\frac{\log \left| \det H \right|  - \frac{1}{2}\log N! + \frac{1}{4}\log N}{\sqrt{\log N}}
			\to \mathscr{N}(0, 1)
		\end{equation}
	when $H$ is drawn from the GOE.
	Proofs of these central limit theorems also appear in \cite{TaoVu2012, CaiLiaZho2015, BorLaC2015,
	EdeLaC2015}. For related results concerning other models of random matrices, see \cite{Rou2007} and the references therein.\\

	\noindent While the authors of \cite{TaoVu2012} give their own proof of (\ref{GUE CLT}) and (\ref{GOE CLT}), their main interest is
	to establish such a result in the more general setting of Wigner matrices. 
	Indeed, they show that in (\ref{GOE CLT}), we may replace $H$ by $W$, a Wigner matrix whose entries' first
	four moments match those of $\mathscr{N}(0,1)$.
	They also prove the analogous result in the complex case.
	In this paper, we will relax this four moment matching assumption to a 
	two moment matching assumption (see Theorem \ref{main theorem}). \\
	
	\noindent Finally, we mention that new interest in averages of determinants of random (Hermitian) matrices 
	has emerged from the study of complexity of high-dimensional landscapes \cite{FyoWil2007,AufBenCer2013}.

	\subsection{Statement of results: The determinant. }
This subsection gives our main result and suggests extensions in connection with the general class of log-correlated random fields. Our theorems apply to Wigner matrices as defined below.

	\begin{definition}
	\label{def wigner}
        	A complex Wigner matrix, $W = \left(w_{ij}\right)$, is an $N\times N$ 
	Hermitian matrix with entries  
		\[ W_{ii} = \sqrt{\frac{1}{N}}\,x_{ii},\, i = 1, \dots, N, \quad W_{ij} = \frac{1}{\sqrt{2N}} \left(x_{ij} + {\rm i} y_{ij}\right), \, 
		1 \leq i < j \leq N.\]
	Here 
	$\{x_{ii}\}_{1\leq i \leq N}$, $\{x_{ij}\}_{1 \leq i < j \leq N}$, $\{y_{ij}\}_{1 \leq i <  j \leq N}$ are 
	independent identically distributed random variables satisfying
		$
		\mathbb{E} \left(x_{ij}\right) = 0, \mathbb{E}\left(x_{ij}^2\right) = \mathbb{E}\left(y_{ij}^2\right) = 1. 
		$
	We assume further 
        	that the common distribution $\nu$ of $\{x_{ii}\}_{1\leq i \leq N}$, $\{x_{ij}\}_{1 \leq i < j \leq N}$, $\{y_{ij}\}_{1 \leq i <  j \leq N}$, has subgaussian decay, i.e.\ there exists 
        	$\delta_0 > 0$ such that
        		\begin{equation}
		\label{subgaussian} 
		\int_\mathbb{R} e^{\delta_0 x^2}{\rm d}\nu(x) < \infty. 
		\end{equation}
	In particular, this means that all the moments of the entries of the matrix are bounded.
	In the special case $\nu = \mathscr{N}(0,1)$, 
	$W$ is said to be drawn from the Gaussian Unitary Ensemble (GUE). 
	
	Similarly, we define a real Wigner matrix to have entries of the form
		$ W_{ii} = \sqrt{\frac{2}{N}} x_{ii}$, $W_{ij} = \sqrt{\frac{1}{N}} x_{ij}$, 
	where $\left\{x_{ij}\right\}_{1\leq i, j \leq N}$ are independent identically distributed random variables satisfying
		$ \mathbb{E}\left(x_{ij}\right) = 0, \mathbb{E}\left(x^2_{ij}\right) = 1.$
	As in the complex case, we assume
	the common distribution $\nu$ satisfies (\ref{subgaussian}).
	In the special case $\nu = \mathscr{N}(0,1)$,
	$W$ is said to be drawn from the Gaussian Orthogonal Ensemble (GOE).
	\end{definition}
			
	\noindent Our main result extends (\ref{GUE CLT}) and  (\ref{GOE CLT}) to the above class of Wigner matrices.  In particular,
	this answers a  conjecture from \cite[Section 8]{TaoVu2014}, which asserts that the central limit theorem 
	(\ref{GOE CLT})
	holds for Bernoulli ($\pm 1$) matrices. Note that in the following statement, our centering
	differs from (\ref{GUE CLT}) and (\ref{GOE CLT}) because we normalize our matrix entries
	to have variance of size $N^{-1}$.

	\begin{theorem}\label{main theorem}
		Let $W$ be a real Wigner matrix satisfying (\ref{subgaussian}). Then
		\begin{equation}\label{main}
		\frac{\log\left| \det W\right| + \frac{N}{2} }{\sqrt{\log N}} \to \mathscr{N}(0, 1).  
		\end{equation}
		If $W$ is a complex Wigner matrix satisfying (\ref{subgaussian}), then
		\begin{equation}
			\frac{\log\left| \det W \right| + \frac{N}{2}}{\sqrt{\frac{1}{2}\log N}} 
			\to \mathscr{N}(0, 1). 
		\end{equation}
	\end{theorem}
	
	 \noindent Assumption (\ref{subgaussian}) may probably  be relaxed to a finite moment assumption, but we will not pursue this direction here. Similarly, it is likely that the matrix entries do not need to be identically distributed;
	 only the first two moments need to match. However we consider the case of a 
	 unique $\nu$ in this paper.

	\begin{remark}\label{rem:multidim}
	Let $H$ be drawn from the GUE normalized so that in the limit as $N\to\infty$, the distribution
	of its eigenvalues is supported on $[-1,1]$, and let
		\[D_N(x) = -\log\left| \det\left( H - x\right) \right|. \]
	In \cite{Kra2007}, Krasovsky proved that  for $x_k \in (-1,1)$, $k =1, \hdots, m$, 
	$x_j \neq x_k$, uniformly in $\Re \left(\alpha_k\right) > - \frac{1}{2}$, $\mathbb{E}\left( e^{-\sum_{k=1}^m \alpha_k D_N\left(x_k\right)} \right)$ is asymptotic to
		\begin{align}\label{krasovsky}
			\prod_{k=1}^m \left( C\left(\frac{\alpha_k}{2}\right) \left(1-x_k^2\right)^{\frac{\alpha_k^2}{8}} 
			 N^{\frac{\alpha_k^2}{4}} e^{\frac{\alpha_k N}{2}\left( 2x_k^2 -1 - 2\log 2 \right)}\right)
			 &\prod_{1\leq \nu < \mu \leq m} \left(2\left| x_\nu - x_\mu \right| \right)^{-\frac{\alpha_\nu \alpha_\mu}{2}}
			 \left(1 + {\rm O}\left( \frac{\log N}{N} \right) \right),
		\end{align}
	as $N\to\infty$. Here $C(\cdot)$ is the Barnes function.
	Since the above estimate holds uniformly for $\Re \left(\alpha_k\right) > -\frac{1}{2}$,  
	(\ref{krasovsky}) shows that
	letting
		\[	
			\wt D_N(x) = \frac{D_N(x) - N\left( x^2 - \frac{1}{2} - \log 2 \right)}{\sqrt{\frac{1}{2} \log N}},
		\]
	the vector $\left( \wt D_N\left(x_1\right), \hdots, \wt D_N\left(x_m\right) \right)$ converges in distribution 
	to a collection of $m$ independent standard Gaussians. Our proof of Theorem
	\ref{main theorem} automatically extends this result to Hermitian Wigner matrices as defined above.
	If one were to prove an analogous convergence for the GOE, our proof of Theorem
	\ref{main theorem} would extend the result to real symmetric Wigner matrices as well. 
	\end{remark}

\begin{remark} \label{rem:multidim2}
	We note that (\ref{krasovsky}) was proved for fixed, distinct $x_k$'s. If  (\ref{krasovsky}) holds for collapsing
	$x_k$'s, this means that fluctuations of the log-characteristic polynomial of the GUE  become log-correlated for large dimension, 
	as in the case of the Circular Unitary Ensemble \cite{BourgadeZeta}. More specifically, let
	$\wt D_N(\cdot)$ be as above, and let $\Delta$ denote the distance between two points $x, y$ in
	$(-1, 1)$. For $\Delta \geq 1/N$, we expect the covariance between $\wt D_N(x)$ and $\wt D_N(y)$
	to behave like $\frac{\log\left(1/\Delta\right)}{\log N}$, and for $\Delta \leq 1/N$, we expect it to
	converge to 1.
\end{remark}

\noindent 
	Our method  automatically
	 establishes the content of Remark \ref{rem:multidim2} for Wigner matrices,
	 conditional on the knowledge of GOE and GUE cases. 
	 The exact statement is as follows, and we omit the proof, strictly similar to Theorem \ref{main theorem}. Denote
	 	$$
	 	L_N(z)=\log|\det(W-z)|-N\int_{-2}^2\log|x-z|\,\rd \rho_{\rm sc}(x).
	 	$$

\begin{theorem}\label{multldim}
		Let $W$ be a real Wigner matrix satisfying (\ref{subgaussian}). Let 
		$\ell\geq 1$, $\kappa>0$ and let $E^{(1)}_N)_{N\geq 1},\dots$, $(E^{(\ell)}_N)_{N\geq 1}$
		be energy levels in $[-2+\kappa,2-\kappa]$.
		Assume that for all $i\neq j$, for some constants $c_{ij}$ we have
$$
\frac{\log |E_N^{(i)}-E_N^{(j)}|}{-\log N}\to c_{ij}\in[0,\infty]
$$		
as $N\to\infty$. Then
\begin{equation}\label{eqn:logcor}
\frac{1}{\sqrt{\frac{1}{2}\log N}}\left(
L_N\left( \left(E^{(1)}_N\right) \right),\dots,
L_N\left(\left(E^{(\ell)}_N\right)\right)
\right)
\end{equation}
converges in distribution to a Gaussian vector with covariance $ (\min(1,c_{ij}))_{1\leq i,j\leq N}$ (with diagonal $1$ by convention), provided the same result holds for GOE. \\

\noindent The same result holds for Hermitian Wigner matrices, assuming it is true in the GUE case, up to a change in the normalization from $\sqrt{\frac{1}{2}\log N}$ to $\sqrt{\log N}$ in (\ref{eqn:logcor}).
\end{theorem}	 	
	 	
	 	\noindent Theorem \ref{multldim} says $L_N$ converges to a log-correlated field, provided this result holds for the Gaussian ensembles. It therefore suggests that the universal limiting behavior of extrema and convergence to Gaussian multiplicative chaos conjectured for unitary matrices in \cite{FyoHiaKea12} extends to the class of Wigner matrices. 
Towards these conjectures, \cite{ArgBelBou15,ChaMadNaj16,PaqZei17,FyoSim2015,LamPaq16} proved asymptotics on the maximum of characteristic polynomials of circular unitary and invariant ensembles, and \cite{BerWebWon2017,NikSakWeb2018,Web2015} established convergence to the Gaussian multiplicative chaos, for the same models.
We refer to  \cite{Arg16} for a survey on log-correlated fields and their connections with random matrices, branching processes, the Gaussian free field, and analytic number theory.

\subsection{Statement of results: Fluctuations of Individual Eigenvalues. }	
With minor modifications, the proof of Theorem \ref{main theorem} also extends the results of \cite{Gus2005} and \cite{ORo2010} 
which describe the fluctuations of individual eigenvalues in the GUE and GOE cases, respectively. 
By adapting the method of \cite{TaoVu2011}, \cite{ORo2010} proves the following theorem under
the assumption that the first four moments of the matrix entries match those of a standard Gaussian.
In Appendix B, we show that the individual eigenvalue fluctuations of the GOE (GUE) also hold for real (complex) Wigner matrices in the sense of Definition \ref{def wigner}.
In particular, the fluctuations of eigenvalues of Bernoulli matrices are Gaussian in the large dimension limit, which answers a question from \cite{TaoVu2014}. \\

\noindent To state the following theorem, we follow the notation of Gustavsson \cite{Gus2005} and write 
$k(N) \sim N^\theta$ to mean that $k(N) = h(N)N^\theta$ where $h$ is a function such that for 
all $\e > 0$, for large enough $N$,
	\begin{equation}
	\label{gustavsson sim}
	N^{-\e}\leq h(N)\leq N^\e.
	\end{equation}
	In the following, $\gamma_k$ denotes the $k$\textsuperscript{th} quantile of the semicircle law,
			\begin{equation}\label{quantiles}
			 \frac{1}{2\pi}\int_{-2}^{\gamma_k} \sqrt{\left(4 - x^2 \right)_+}{\rm d}x = \frac{k}{N}.
			\end{equation}

\begin{theorem}
\label{fluctuations of individual eigenvalues}
	Let $W$ be a Wigner matrix satisfying (\ref{subgaussian}) with eigenvalues $\lambda_1 < \lambda_2 <
	\dots < \lambda_N$. Consider $\left\{ \lambda_{k_i} \right\}_{i=1}^m$ such that  $0 < k_{i+1} - k_{i} \sim N^{\theta_i},
	0 < \theta_i \leq 1$, and $k_i /N \to a_i \in (0,1)$ as $N \to \infty$.
	Let
	\begin{equation}
	\label{def: X_i}
		X_i = \frac{\lambda_{k_i} - \gamma_{k_i}}{\sqrt{ \frac{4\log N}{\beta \left(4 - \gamma_{k_i}^2 \right) N^2}  }},
		\quad i = 1, \dots, m,
	\end{equation}
	with $\beta =1$ for real Wigner matrices, and $\beta = 2$ for complex Wigner matrices.
	Then as $N \to \infty$,
	\[
		\P\left\{ X_1 \leq \xi_1, \dots, X_m \leq \xi_m \right\} \to \Phi_{\Lambda}\left(\xi_1, \dots, \xi_m \right),
	\]
	where $\Phi_\Lambda$ is the cumulative distribution function for the $m$-dimensional normal distribution with
	covariance matrix $\Lambda_{i,j} = 1 - \max\left\{ \theta_k: i \leq k < j < m \right\}$ if 
	$i < j$, and $\Lambda_{i,i} = 1$.
\end{theorem}

\noindent The above theorem  has been known to follow from the  homogenization result in \cite{BouErdYauYin2016} (this technique gives a simple expression for the relative {\it individual} positions of coupled eigenvalues from GOE and Wigner matrices) and fluctuations of mesoscopic linear statistics; see \cite{LanSos2018} for a proof of eigenvalue fluctuations for Wigner and invariant ensembles.  However, the technique from  \cite{BouErdYauYin2016} is not enough for Theorem \ref{main theorem}, as the determinant depends on the positions of {\it all} eigenvalues.

	\subsection{Outline of the proof. }\  In this section, we give the main steps of
	the proof of Theorem  \ref{main theorem}. Our outline discusses the real case, but
	the complex case follows the same scheme. \\
	
	\noindent The main conceptual idea of the proof follows the three step strategy of 
	\cite{ErdPecRmSchYau2010,ErdSchYau2011}. With a priori localization of eigenvalues (step one, \cite{ErdYauYin2012Rig, CacMalSch2015}), one can prove that the determinant has universal fluctuations after a adding a small Gaussian noise (this second step relies on a stochastic advection equation from \cite{BourgadeExtreme}). The third step proves by a density argument that the Gaussian noise does not change the distribution of the determinant, thanks to a perturbative moment matching argument as in \cite{TaoVu2011,ErdYauYin2012Bulk}. We include Figure \ref{diagram of implications} below to 
	help summarize
	the argument.\\

	\noindent {\it First step: small regularization.}
	In Section \ref{smoothing section}, 
	with Theorems \ref{schlein} and \ref{fixed energy thm}, we reduce the proof of Theorem \ref{main theorem} to showing the convergence
\begin{equation}\label{eqn:probaconv}
 \frac{\log\left| \det (W+\ii\eta_0)\right| + c_N}{\sqrt{\log N}} \to \mathscr{N}(0, 1)
				\end{equation}
	with  some explicit deterministic $c_N$, and the small regularization parameter
		\begin{equation}
		\label{def:eta0}
			\eta_0 = \frac{e^{ \left(\log N\right)^{ \frac{1}{4}  } }}{N}.
		\end{equation}

\noindent {\it Second step: universality after coupling.} 
Let $M$ be a symmetric matrix which serves as the initial condition for the
	matrix Dyson's Brownian Motion (DBM) given by 
		\begin{equation}
			\label{eq:DBM evolution} 
		\rd M_t = \frac{1}{\sqrt{N}}{\rm d}B^{(t)} - \frac{1}{2}M_t{\rm d}t.
		\end{equation}
Here $B^{(t)}$ is a symmetric $N \times N$ matrix such that $B^{(t)}_{ij}\left(i < j\right)$ and $B_{ii}^{(t)}/\sqrt{2}$ are
independent standard Brownian motions.
	The above matrix DBM induces 
	a collection of independent standard Brownian motions (see \cite{AndGuiZei2010}), $\tilde{B}^{(k)}_t/\sqrt{2}$, $k = 1, \dots, N$ such that the eigenvalues of $M$ satisfy  the system of stochastic differential equations
	\begin{align}
		\label{eqn:DBM}
		{\rm d}x_k(t) &= \frac{{\rm d}\tilde{B}_{t}^{(k)}}{\sqrt{N}} + \left( \frac{1}{N} \sum_{l \neq k}
			\frac{1}{x_k(t) - x_l(t)} - \frac{1}{2} x_k(t)\right){\rm d}t 	
	\end{align}
	with initial condition given by the eigenvalues of $M$.
	It has been known since \cite{McK1969} that the system 
	(\ref{eqn:DBM}) has a unique strong solution.
	With this in mind, we follow \cite{BouErdYauYin2016} and introduce the following
	coupling scheme. First, run the matrix DBM taking $\tilde{W}_0$, a Wigner matrix, as the initial condition. Using
	the induced Brownian motions, run the dynamics given by (\ref{eqn:DBM}) using the eigenvalues 
	$y_1 < y_2 < \dots < y_N$
	of $\tilde{W}_0$ as the initial condition. Call the solution to this system ${\bm y}(\tau)$.
	Using the very same (induced) Brownian motions, run
	the dynamics given by (\ref{eqn:DBM}) again, this time using the eigenvalues of a GOE matrix, ${\bm x}(0)$,
	as the initial condition. Call the solution to this system ${\bm x}(\tau)$. \\

	\noindent Now fix $\epsilon > 0$ and let 
	\begin{equation}
		\label{def:tau}
			\tau = N^{-\epsilon}.
		\end{equation}
	Using Lemma \ref{extreme}, we show that 
		\begin{equation}
		\label{coupling at tau}
		 \frac{\sum_{k = 1}^N \log \left|x_k(\tau) + {\rm i}\eta_0\right| - 
				\sum_{k = 1}^N \log \left|y_k(\tau) + {\rm i}\eta_0\right| }{\sqrt{\log N}}
		\end{equation}
	and
		\begin{equation} \label{smoothing t} \frac{\sum_{k = 1}^N \log \left|x_k(0) + z_\tau\right| - 
				\sum_{k = 1}^N \log \left|y_k(0) + z_\tau\right| }{\sqrt{\log N}}
		\end{equation}
	are very close. Here $z_\tau$ is as in (\ref{def:ztau}) with $z = {\rm i} \eta_0$. 
	The  significance of this is that since
	$z_\tau \sim \ii\tau$, we can use
	Lemma \ref{expectation} and well-known central limit theorems which apply to nearly macroscopic scales 
	to show that (\ref{smoothing t}) has variance of order $\e$. Consequently,
	(\ref{coupling at tau}) is also small, and since ${\bm x}(\tau)$ is distributed as the eigenvalues of a GOE matrix,
	we have proved universality of the regularized determinant after coupling.\\

	\noindent {\it Third step: moment matching}. In Section \ref{conclusion}, we conclude the proof 
	of Theorem \ref{main theorem}. First, we choose
	$\tilde{W}_0$ so that $\tilde{W}_\tau$ and $W$ have entries whose first four moments are
	close, as in \cite{ErdYauYin2012Bulk}. With this approximate moment matching, we use a perturbative argument, as in \cite{TaoVu2012},
	to prove that (\ref{eqn:probaconv})
	holds for $W$ if and only if it holds for $\tilde{W}_\tau$. 	
	But as (\ref{coupling at tau}) is small, this means (\ref{eqn:probaconv})
	holds for $W$ if and only if it holds for a GOE matrix. By (\ref{GOE CLT}), this concludes the proof.

	\begin{wrapfigure}[22]{l}{8cm}
	\vspace{-1.2cm}
	\begin{center}
	\begin{tikzcd}[row sep = 3cm, column sep = 3cm] 
		& W \\
		\tilde{W}_0 \arrow{r} \arrow[rightarrow]{r} [anchor=center,yshift=2ex] 
			{\text{Matrix DBM } \rd B_{ij}} & \tilde{W}_\tau  \arrow[leftrightarrow]{u} [anchor=center,rotate=-90,yshift=2ex]{\text{Moment Matching (3)}}%& W 
		\\ [-85pt]
		{\bm y(0)} \arrow{r} [anchor=center,yshift=-2ex]{\text{Eigenvalues DBM } \rd \tilde{B}_k}& {\bm y(\tau)} \\ 
		[-30pt]
		{{\bm x}(0)} \arrow[rightarrow]{r}[anchor=center,yshift=2ex]{ \text{Eigenvalues DBM }\rd \tilde{B}_k}& {{\bm x}(\tau)}
		\arrow{u} [anchor=center,rotate=-90,yshift=2ex]{\text{Coupling (2)}} 	
	\end{tikzcd} 
	\caption{
       		We show (\ref{main}) holds for $\tilde{W}_\tau$ if and only
		if it holds for $W$, and
		we prove that (\ref{main}) holds for
		${\bm x}(\tau)$ if and only if (\ref{main}) holds for $\tilde{W}_\tau$. 
		Since ${\bm x}(\tau)$ is distributed as the eigenvalues of a GOE matrix, it satisfies
		(\ref{GOE CLT}) and we conclude the proof.
		Note that $\log| \det \tilde{W}_\tau | = \sum \log \left|y_k(\tau) \right|$ pathwise because
		$B$ induces $\tilde{B}$. 
    } \label{diagram of implications}
	\end{center}
\end{wrapfigure}
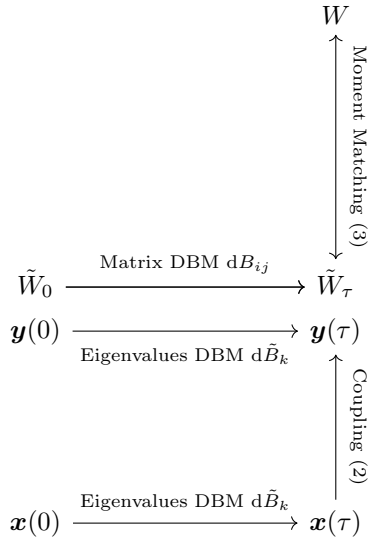

	\subsection{Notation. }
	We shall make frequent use of the notations $s_W$ and $m_{sc}$ in the remainder of this paper. We 
	state their definitions here for easy reference.
	Let $W$ be a Wigner matrix with eigenvalues $\lambda_1 < \lambda_2 < \dots < \lambda_N$.
	For $\Im (z) > 0$, define
		\begin{equation}
		\label{def:stransform}
		s_W(z) = \frac{1}{N} \sum_{k=1}^N \frac{1}{\lambda_k - z},
		\end{equation}
	the Stieltjes transform of $W$. 
	Next, let
		\begin{equation}
		\label{def:msc}
			m_{sc}(z) =  \frac{-z + \sqrt{z^2 -4}}{2},
		\end{equation}
	where the square root $\sqrt{z^2 -4}$ is chosen with the branch cut in $[-2, 2]$ so that $\sqrt{z^2-4} \sim z$ 
	as $z \to \infty$. Note that
		\begin{equation}
			\label{self consistent m}
			m_{sc}(z) + \frac{1}{m_{sc}(z)} + z = 0.
		\end{equation} 

\noindent Finally, throughout this paper, unless indicated otherwise,
	$C$ ($c$) denotes a large (small) constant independent of all other parameters of the problem. 
	It may vary from line to line.

\section{Initial Regularization}\label{smoothing section}

	\noindent Let $y_1 < y_2 < \dots <y_N$ denote 
	the eigenvalues of $W$, a real Wigner matrix
	satisfying (\ref{subgaussian}). 
	We first prove 
	we only need to show Theorem \ref{main theorem} for a slight regularization of the logarithm.

	\begin{proposition}\label{smoothing}
		Set
		\begin{equation*}
		g(\eta) = \sum_{k} \left(\log \left| y_k+\ii\eta\right| - \log \left| y_k\right|\right)-\int_0^{\eta}N\Im 
			\left( m_{sc}(\ii s)\right) \rd s
		\end{equation*}
		and recall
		$ \eta_0 =\frac{e^{ \left(\log N\right)^{ \frac{1}{4}  } }}{N} $
		as in (\ref{def:eta0}).
		Then we have the convergence in probability
	$$\frac{g(\eta_0)}{\sqrt{\log N}}\to 0.$$
	\end{proposition}
	\noindent To prove Proposition \ref{smoothing}, we will use Theorems \ref{schlein} and \ref{fixed energy thm}
	as input. In \cite{CacMalSch2015}, Theorem \ref{schlein} is stated for complex Wigner matrices, however,
	the argument there proves the same statement for real Wigner matrices.
	\begin{theorem}[Theorem 1 in \cite{CacMalSch2015}]\label{schlein}
		Let $W$ be a Wigner matrix and
		fix $\tilde{\eta} > 0$. 
		For any $\tilde{E} > 0$, there exist constants $M_0, N_0, C, c, c_0 >0$ such that 
					\[ \mathbb{P}\left( \left| \Im \left(s_W\left(E + {\rm i}\eta\right)\right) - 
					\Im \left(m_{sc}\left(E + {\rm i}\eta\right)\right) \right| 
						\geq \frac{K}{N\eta} \right) 
						\leq \frac{\left(Cq\right)^{cq^2}}{K^q} \]
		for all $\eta \leq \tilde{\eta}$, $|E| \leq \tilde{E}$, $K > 0$, $N > N_0$ such that $N\eta > M_0$,
		and
		$q \in \mathbb{N}$ with $q \leq c_0\left(N\eta\right)^{\frac{1}{8}}$.
	\end{theorem}	
	\begin{remark}
		In  \cite{ErdYauYin2012Rig}, the authors proved that for some positive constant $C_0$, and $N$ large enough,
			\[ \left| s_W\left(E + {\rm i}\eta\right) - m_{sc}\left(E + {\rm i}\eta\right) \right| \leq
				\frac{e^{C_0(\log\log N)^2}}{N\eta} \] 
		holds with high probability. Though this estimate is weaker than the estimate
		of Theorem \ref{schlein}, it holds for a more general model of Wigner matrix in which
		the entries of the matrix need not have identical variances. On the other hand, we require the  
		stronger estimate in Theorem \ref{schlein} in our proof of Proposition
		\ref{smoothing}, and so we restrict ourselves to Wigner matrices as defined
		in Definition \ref{def wigner}.
		The proof of Lemma \ref{expectation} also relies on Definition \ref{def wigner}. 
	\end{remark}
	\begin{theorem}[Theorem 2.2 in \cite{BouErdYauYin2016}]\label{fixed energy thm}
		Let $\rho_1$ denote the first correlation function for the eigenvalues of an $N\times N$ Wigner matrix,
		and let $\rho(x) = \frac{1}{2\pi} \sqrt{\left( 4 - v^2 \right)_+}$. 
		Then for any $F: \mathbb{R} \to \mathbb{R}$ continuous and compactly supported, and for any $\kappa > 0$,
		we have, 
			\begin{equation}\label{FEU} \lim_{N\to\infty} \sup_{E \in [-2 + \kappa, 2 - \kappa]} \left |\frac{1}{\rho(E)} 
			\int F(v)\rho_1\left( E + \frac{v}{N\rho(E)}\right){\rm d}v -
			\int F(v) \rho(v)\, {\rm d}v \right| = 0.\end{equation}
	\end{theorem}
	\begin{remark}
		In fact Theorem 2.2 in \cite{BouErdYauYin2016} makes a much stronger statement, namely it states the
		analogous convergence for all correlation functions in the case of generalized Wigner matrices.
	\end{remark}
	
	\begin{corollary}\label{repulsion}
		For any small fixed $\kappa,\gamma > 0$ there exists $C,N_0>0$ such that for any $N\geq N_0$ and any interval $I \subset [-2 + \kappa, 2 - \kappa]$ we have
			\[ \mathbb{E}\left( \left| \left\{ y_k: y_k \in I \right\} \right|\right)\leq CN |I| + \gamma. \]
				\end{corollary}
	\begin{proof}
		In Theorem \ref{fixed energy thm}, choosing $F$ to be an indicator of an interval of length
		$1$ gives an expected value ${\rm O}(1)$.
		Since the statement of Theorem \ref{fixed energy thm} holds uniformly in $E$, 
		we may divide the interval $I$ into sub-intervals of length order $1/N$
		to conclude. 
	\end{proof}

	\begin{corollary}\label{micro fixed energy}
		Let $E\in[-2+\kappa,2-\kappa]$ be fixed and $I_\beta = (E -  \beta/2, E + \beta/2)$ with $\beta=\oo(N^{-1})$. Then
			\[ \lim_{N\to\infty} \mathbb{P}\left( \left| \left\{ y_k \in I_\beta \right\}\right| = 0 \right) = 1. \]
	\end{corollary}
	\begin{proof}
		Let $\e$ be any fixed small constant.
		Let $f$ be fixed, smooth, positive, equal to $1$ on $[-1,1]$ and $0$ on $[-2,2]^c$.
		Then
		$$
				\mathbb{P}\left( \left| \left\{ y_k \in I_\beta \right\}\right| \geq 1 \right) \leq
				\mathbb{E} \left( \left| \left\{ y_k \in I_\beta \right\}\right| \right) \leq
				\mathbb{E} \left( \sum_k f\left(N(y_k-E)/\e\right) \right) \leq 10\e,
		$$
		where the last bound holds for large enough $N$ by Theorem \ref{fixed energy thm}. 
	\end{proof}
	\begin{proof}[Proof of Proposition \ref{smoothing}]
	We first choose $\tilde{\eta}<\eta_0$ so that we can use Theorem \ref{schlein} to estimate
		\[\mathbb{E}\left( \left| g\left(\eta_0\right) - g\left(\tilde{\eta}\right) \right| \right),\]
	and then take care of the remaining error using Corollaries \ref{repulsion} and \ref{micro fixed energy}.
	Let
		\[ \tilde{\eta} = \frac{d_N}{N}, \ \ {\rm with} \ \ d_N = \left( \log N \right)^{\frac{1}{4}},\]
	and observe that
		\begin{equation}\label{eqn:inter1} \mathbb{E}\left( \left| g\left(\eta_0\right) - g\left(\tilde{\eta}\right) \right| \right)= 
		  \mathbb{E}\left( \left|\int_{\tilde{\eta}}^{\eta_0} N \Im \left( s_W(\ii t) -m_{sc}(\ii t)\right)
		  	{\rm d}t \right| \right)
		  \leq 
		  \int_{\tilde{\eta}}^{\eta_0} \mathbb{E} (N\left|\Im \left( s_{W_1}\left({\rm i}t\right) - 
		  	m_{sc} ({\rm i}t) \right|\right))\, {\rm d}t. \end{equation}
	In estimating the right hand side above, we will use the notation
		\[ \Delta(t) = \left|\Im \left( s_{W_1}({\rm i}t) - 
		  	m_{sc}({\rm i}t \right) \right|. \]
	For $N$ sufficiently large, by Theorem \ref{schlein} with $q =2$, we can write the right hand side of (\ref{eqn:inter1}) as
		\begin{multline}
			\int_{\tilde{\eta}}^{\eta_0} \int_0^\infty \mathbb{P}\left(N  \Delta\left( t \right) > u\right){\rm d}u{\rm d}t 
			= \int_{\tilde{\eta}}^{\eta_0}\left( 
				\int_0^1 \mathbb{P}\left( \Delta\left( t \right) > \frac{K}{Nt}\right)\,\frac{{\rm d}K}{t} +
				\int_{1}^\infty \mathbb{P}\left( \Delta\left( t \right) > \frac{K}{Nt}\right)\frac{{\rm d}K}{t} \right){\rm d}t \\
			\leq \int_{\tilde{\eta}}^{\eta_0} \left( \frac{1}{t} + \int_1^\infty \frac{C}{K^2} \frac{dK}{t} \right){\rm d}t
			\leq\left(1 + C\right)\log\left( \frac{\eta_0}{\tilde{\eta}}\right) 
			={\rm o}\left( \sqrt{\log N} \right).\label{est1}
		\end{multline}
	Next we estimate
	 $\sum_{k} \left(\log \left| y_k + {\rm i}\tilde{\eta}\right| - \log \left| y_k\right|\right)$, and 
	this will give us a bound for $\mathbb{E}\left(\left| g(\tilde{\eta}) \right|\right)$. 
	Taylor expansion yields
		\[ \sum_{|y_k| > \tilde{\eta}}  \left(\log \left| y_k + {\rm i}\tilde{\eta}\right| - \log \left| y_k\right|\right) \leq
		\sum_{|y_k| > \tilde{\eta}} \frac{\tilde{\eta}^2}{y_k^2}.
 \]
	Define
		$N_1(u) = \left| \left\{ y_k : \tilde{\eta} \leq |y_k| \leq u \right\}\right|$.
	Using integration by parts and Corollary \ref{repulsion}, we have
		\begin{equation} \mathbb{E} \left( \sum_{|y_k| > \tilde{\eta}} \frac{\tilde{\eta}^2}{y_k^2} \right) =
			\mathbb{E}\left( \int_{\tilde{\eta}}^\infty \frac{\tilde{\eta}^2}{y^2}\,{\rm d}N_1(y) \right)
			= 2\tilde{\eta}^2 \int_{\tilde{\eta}}^\infty \frac{\mathbb{E}\left(N_1(y)\right)}{y^3}\,{\rm d}y = {\rm O}\left (d_N\right).\label{eqn:est2}
			\end{equation}
	We now estimate $\sum_{|y_k| \leq \tilde{\eta}}  \left(\log \left| y_k + {\rm i}\tilde{\eta}\right| - \log \left| y_k\right|\right)$.
	We consider two cases.
	First, let $A_N = b_N/N$ for some very small $b_N$, for example 
		\[ b_N = e^{-\left(\log N\right)^{\frac{1}{4}}}.\]
	For $u > 0$ we denote
		$ N_2(u) = \left| \left\{ y_k \,:\, A_N < | y_k |\leq u \right\} \right|$. Then again using integration by parts and Corollary \ref{repulsion}  we obtain
		\begin{multline*}
			\mathbb{E}\left(
			\sum_{A_N <| y_k| < \tilde{\eta}}  \left(\log \left| y_k + {\rm i}\tilde{\eta}\right| - \log \left| y_k\right|\right)
			\right)
			= \mathbb{E}\left( \int_{A_N}^{\tilde{\eta}} 
				\left( \log \left| y + {\rm i}\tilde{\eta}\right| - \log \left|y\right|\right){\rm d}N_2(y)  \right)\\
				\leq \log\left(\sqrt{2}\right) \mathbb{E}\left( N_2\left( \tilde{\eta} \right) \right) + 
				\int_{A_N}^{\tilde{\eta}} \frac{\mathbb{E}\left(N_2(y)\right)}{y}{\rm d}y
				= {\rm O}\left(d_N + d_N\log\left( \frac{d_N}{b_N} \right)\right) = \oo\left( \sqrt{\log N}\right).
		\end{multline*}
It remains
	to estimate $\sum_{|y_k| < A_N}\left( \log\left| y_k + {\rm i}\tilde{\eta}\right| - \log\left|y_k\right|\right)$. 
	By Corollary \ref{micro fixed energy}, we have
		\begin{equation}\label{est3}
		 \mathbb{P}\left(\sum_{|y_k| < A_N}\left( \log\left| y_k + {\rm i}\tilde{\eta}\right| - \log\left|y_k\right| \right) = 0\right) 
		 	\geq \mathbb{P}\left( \left|\left\{ y_k \in [-A_N,A_N]\right\} \right|= 0 \right) \to 1.
		\end{equation}
	\noindent 
	The estimates (\ref{est1}) and (\ref{eqn:est2}) along with Markov's inequality, and the bound (\ref{est3}),
	conclude the proof.
	\end{proof}
\section{Coupling of Determinants}\label{DBM}
	
	\noindent In this section, we use the coupled Dyson Brownian Motion introduced in \cite{BouErdYauYin2016} to
	compare (\ref{smoothing t}) and (\ref{coupling at tau}). 
	Define $\tilde{W}_\tau$ by running the matrix Dyson Brownian Motion
	(\ref{eq:DBM evolution}) with initial condition $\tilde{W}_0$
	where $\tilde{W}_0$ is a Wigner matrix with eigenvalues ${\bm y}$.
	Recall that this induces a collection of Brownian motions
	$\tilde{B}^{(k)}_t$ 
	so that the system (\ref{eqn:DBM}) with initial condition ${\bm y}$
	has a (unique strong) solution 
	${\bm y}(\cdot)$, and ${\bm y}(\tau)$ are the eigenvalues of $\tilde{W}_\tau$. Using the same (induced) Brownian motions as we used to define ${\bm y}(\tau)$, define
	${\bm x}(\tau)$ by running the dynamics (\ref{eqn:DBM}) with initial condition given by the eigenvalues of a GOE matrix.
	Using the result of Section \ref{smoothing section} as an input to Lemma \ref{extreme}, we now prove Proposition
	\ref{prop:advection} which says that (\ref{coupling at tau}) and (\ref{smoothing t})  are asymptotically equal in law. \\

	\noindent To study the coupled dynamics of ${\bm x}(t)$ and ${\bm y}(t)$, we
	follow \cite{LanSosYau2016, BourgadeExtreme}. For $\nu \in [0,1]$, let
		\begin{equation}
		\label{def:initialcondition}
			\lambda^\nu_k(0) = \nu x_k + \left(1 - \nu\right)y_k
		\end{equation}
	where ${\bm x}$ is the spectrum of a GOE matrix, and ${\bm y}$ is the spectrum of $\tilde{W}_0$.
	With this initial condition, we denote the (unique strong) solution to (\ref{eqn:DBM}) by ${\bm \lambda}^{(\nu)}(t)$.
	Note that
$
			{\bm \lambda}^{(0)}(\tau) = {\bm y}(\tau)$ and $
			{\bm \lambda}^{(1)}(\tau) = {\bm x}(\tau).
$ 
Let
		\begin{equation}
		\label{eqn:ft}
		f^{(\nu)}_t(z) =e^{-\frac{t}{2}} \sum_{k=1}^N \frac{u_k(t)}{\lambda^{(\nu)}_k(t)-z}, 
		\quad  u_k(t) = \frac{\rd}{\rd\nu}\lambda^{(\nu)}_k(t),
		\end{equation}
	(see \cite{LanSosYau2016} for existence of this derivative) and observe that
		\begin{equation}
		\label{main obs}
		\frac{\rd}{\rd\nu} \sum_k \log\left| \lambda_k^{(\nu)}(t) - z \right| = e^{\frac{t}{2}}\Re  \left(f_t(z)\right).
		\end{equation}
	Lemma \ref{extreme} below from \cite[Proposition 3.3]{BourgadeExtreme}, tells us that we may estimate
	$f_\tau(z)$ by $f_0\left(z_\tau\right)$, with $z_\tau$ as in (\ref{def:ztau}) and $\tau$ as in (\ref{def:tau}). 
	
	\begin{lemma}
		\label{extreme}
		There exists $C_0>0$ such that with $\varphi=e^{C_0(\log\log N)^2}$, 
		for any $\nu\in[0,1]$, $\kappa>0$ (small) and $D>0$ (large),  
		 there exists $N_0(\kappa,D)$ so that for any $N\geq N_0$ we have
			\[ \mathbb{P}\left(\left| f^{(\nu)}_t(z) - f^{(\nu)}_0\left(z_t\right) \right| <  \frac{\varphi}{N\eta}\ {\rm for\ all}\ 0<t<1\ {\rm and}\ z=E+\ii \eta, \frac{\varphi}{N}<\eta<1,|E|<2-\kappa\right)
				\geq 1-N^{-D}.\]
		In the above, $z_t$ is given by
			\begin{equation}
			\label{def:ztau} z_t = \frac{1}{2}\left( e^{\frac{t}{2}}\left(z + \sqrt{z^2 - 4}\right) + 
				e^{-\frac{t}{2}}\left(z - \sqrt{z^2 - 4}\right) \right).
			\end{equation}
	\end{lemma}	
	\noindent For $z = {\rm i}\eta_0$, we have
		\begin{equation}
		\label{zt expansion}
		 z_t = {\rm i} \left( \eta_0 + \frac{t \sqrt{\eta_0^2 + 4}}{2} \right) + {\rm O}\left (t^2\right),
		\end{equation}
	and $\eta_0$ is large enough to make use of Lemma \ref{extreme}.
	Therefore,
	integrating both sides of (\ref{main obs}), we have by Lemma \ref{extreme} that with overwhelming probability,
		\begin{multline*}
			\sum_k \left( \log \left|x_k(\tau) + {\rm i}\eta_0\right| - \log \left|y_k(\tau) + {\rm i}\eta_0\right|  \right) = 
			\int_0^1 \frac{\rd}{\rd\nu} \sum_k \log\left| \lambda_k^{(\nu)}(\tau) - z \right| {\rm d}\nu \\
			= e^{\frac{t}{2}} \Re  \int_0^1 f^{(\nu)}_\tau(z) {\rm d}\nu 
			= e^{\frac{t}{2}}\Re  \int_0^1 
			\left( f_0^{(\nu)}\left(z_\tau\right) + {\rm O}\left ( \frac{\varphi}{N\eta_0}\right)\right) {\rm d}\nu 
			= e^{\frac{t}{2}}\Re  \int_0^1 f_0^{(\nu)}\left(z_\tau\right){\rm d}\nu + {\rm o}(1).
		\end{multline*}
		More precisely, the above estimates hold with probability $1-N^{-D}$ for large enough $N$, with rigorous justification by Markov's inequality based on the large moments 
		$\E((\int_0^1 \left( f_\tau^{(\nu)}\left(z_0\right)- f_0^{(\nu)}\left(z_\tau\right)  \right) {\rm d}\nu)^{2p})$, which are bounded by Lemma \ref{extreme}.
	As a consequence, we have proved the following proposition.
	\begin{proposition}
	\label{prop:advection}
	Let $\epsilon >0$, $\tau = N^{-\epsilon}$ and let $z_\tau$ be as in (\ref{def:ztau}) with $z = {\rm i}\eta_0$.
	Then for any $\delta > 0$, 
		\[ \lim_{N\to\infty} \mathbb{P}\left(\left|
		\sum_k \left( \log \left|x_k(\tau) + {\rm i}\eta_0\right| - \log \left|y_k(\tau) + {\rm i}\eta_0\right|  \right) 
		 - \sum_k \left( \log \left|x_k(0) + z_\tau\right| - \log \left|y_k(0) + z_\tau\right|  \right)\right| > \delta  \right) = 0.  \]
	\end{proposition}

\section{Conclusion of the Proof}\label{conclusion}

	\noindent We will conclude the proof of Theorem \ref{main theorem} in the real symmetric case
	in two steps. The first step is to prove a Green's function comparison
	theorem, and the second is to establish Theorem \ref{main theorem} assuming Lemma
	\ref{expectation}, proved  in the Appendix. 

	\subsection{Green's Function Comparison Theorem. }\label{moment matching section}
	In this section, we first use Lemma \ref{moment matching} to 
	choose a $\tilde{W}_0$ so that $\tilde{W}_{\tau}$ given by
	(\ref{eq:DBM evolution})
	and initial condition $\tilde{W}_0$,
	matches $W$ closely up to fourth moment. We will then prove Theorem \ref{4 moment matching theorem},
	which by the result of Section \ref{smoothing section}, says that
	$ \log | \det \tilde{W}_\tau |$ and 
	$ \log \left| \det W \right|$ have the same law as $N\to\infty$.  
	
	\begin{lemma}[Lemma 6.5 in \cite{ErdYauYin2012Bulk}]
		\label{moment matching}
		Let $m_3$ and $m_4$ be two real numbers such that
			\begin{equation}\label{moment condition}
			 m_4 - m_3^2 - 1 \geq 0, \quad m_4 \leq C_2 
			\end{equation}
		for some positive constant $C_2$. Let $\xi^G$ be a Gaussian random variable with mean 
		$0$ and variance $1$. Then for any sufficiently small $\gamma > 0$ (depending on $C_2$),
		there exists a real random variable $\xi$, with subgaussian decay
		and independent of $\xi^G$ such that the first four moments of
			\[ \xi' = \left( 1 - \gamma\right)^{\frac{1}{2}} \xi_\gamma + \gamma^{\frac{1}{2}}\xi^G\]
		are $m_1\left(\xi'\right) = 0$, $m_2\left(\xi'\right) = 1$, $m_3\left(\xi'\right) = m_3$,
		and 
			\[ \left|m_4\left(\xi'\right) - m_4\right| \leq C\gamma \]
		for some $C$ depending on $C_2$. 
	\end{lemma}
	\noindent Now since $\tilde{W}_\tau$ is defined by independent Ornstein-Uhlenbeck processes in each entry, 
	it has the same distribution as
		\[ e^{-\tau/2} \tilde{W}_0 + \sqrt{1-e^{-\tau}}W   \]
	where $W$ is a GOE matrix independent of $\tilde{W}_0$.
	So choosing $\gamma = 1-e^{-\tau}$,
	Lemma \ref{moment matching} says we can
	find $\tilde{W}_0$ so that the first three moments of the entries of $\tilde{W}_\tau$ 
	match the first three moments of the entries of $W$, and the fourth moments of the entries
	of each differ by ${\rm O}(\tau)$.  
	Our next goal is to prove Theorem \ref{4 moment matching theorem} which says that
	with $\tilde{W}_\tau$ constructed this way, if Theorem \ref{main theorem} 
	holds for $\tilde{W}_\tau$, then it holds for $W$. We first introduce stochastic domination and state
	Theorem \ref{local law} which we will use in the proof.
	
	\begin{definition}
		Let $X = \left(X^N(u): N \in \mathbb{N}, u \in U^N\right), Y = \left(Y^N(u): N \in \mathbb{N}, u \in U^N\right)$
		be two families of nonnegative random variables, where $U^N$ is a possibly $N$-dependent parameter set.
		We say that $X$ is stochastically dominated by $Y$, uniformly in $u$, if for every $\epsilon > 0$ and $D > 0$, there exists $N_0(\epsilon, D)$ such that
			\[ \sup_{u \in U^N} \mathbb{P}\left[ X^N(u) > N^\epsilon Y^N(u)\right] \leq N^{-D} \]
		for $N\geq N_0$. Stochastic domination is always uniform in all parameters,
		such as matrix indices and spectral parameters, that are not explicitly fixed.
		We will use the notation $X = O_\prec(Y)$ or $X \prec Y$ for the above property.
	\end{definition}
	
	\begin{theorem}[Theorem 2.1 in \cite{ErdYauYin2012Rig}]\label{local law}
	Let $W$ be a Wigner matrix satisfying (\ref{subgaussian}). Fix $\zeta > 0$ and define the domain
		\[ S = S_N(\zeta) := \left\{ E + {\rm i}\eta \,:\, |E| \leq \zeta^{-1}, \,N^{-1+\zeta} \leq \eta \leq \zeta^{-1} \right\}. \]
	Then uniformly for $i,j = 1,\hdots, N$ and $z \in S$, we have
	\begin{align*}
		 s(z) &= m(z) +  {\rm O}_\prec\left( \frac{1}{N\eta} \right),\\
	G_{ij}(z)& = \left(W -z\right)^{-1}_{ij} = 
		m(z)\delta_{ij} + {\rm O}_\prec\left( \sqrt{\frac{\Im  \left(m(z)\right)}{N\eta}} + \frac{1}{N\eta} \right).
		\end{align*}
	\end{theorem}

	\begin{theorem}\label{4 moment matching theorem}
		Let $F: \mathbb{R} \to \mathbb{R}$ be smooth with compact support, and
		let $W$ and $V$ be two Wigner matrices 
		satisfying (\ref{subgaussian}) such that for $1 \leq i, j \leq N$,
			\begin{numcases}{\mathbb{E}\left(w_{ij}^a\right) = }
				\label{moment matching assumption 3}
				\mathbb{E}\left(v_{ij}^a\right)  & $a \leq 3$ \\
				\label{moment matching assumption 4}
				\mathbb{E}\left(v_{ij}^a\right)+ \OO(\tau) & $a = 4$,
			\end{numcases}
		where $\tau$ is as in (\ref{def:tau}). Further, let $c_N$ be any deterministic sequence and define
			\[
				u_N(W) = \frac{\log | \det \left(W + \ii\eta_0\right)| +c_N}{\sqrt{\log N}}.
			\]
		where $\eta_0$ is as in (\ref{def:eta0}). Then
			\begin{equation}
				\lim_{N\to \infty} \mathbb{E} \left(F\left( u_N(W)\right) - F\left( u_N(V) \right)\right) = 0.\label{eqn:enough}
			\end{equation} 
	\end{theorem}
	\begin{proof}		
		As in \cite{TaoVu2012}, where the authors also used the following technique to analyze
		fluctuations of determinants, we show that the effect of substituting $W_{ij}$ in place of $V_{ij}$ in $V$ is
		negligible enough that making $N^2$ replacements, we conclude the theorem. \\

		\noindent Fix $(i,j)$ and let
		$E^{(ij)}$ be the matrix whose elements are $E^{(ij)}_{kl} = \delta_{ik}\delta_{jl}$.
		Let $W_1$ and $W_2$ be two adjacent matrices in the swapping process described above. 
		Since $W_1, W_2$ differ in just the $(i,j)$ and $(j,i)$ coordinates, we may write
			$$
				W_1 = Q + \frac{1}{\sqrt{N}} U, \ \ \  \ \
				W_2 = Q + \frac{1}{\sqrt{N}} \tilde{U}
			$$
		where $Q$ is a matrix with $Q_{ij} = Q_{ji} = 0$, and
			$$
				U = u_{ij}E^{(ij)} + u_{ji}E^{(ji)}\ \ \ \ \ \ 
				\tilde{U} = \tilde{u}_{ij}E^{(ij)} + \tilde{u}_{ji}E^{(ji)}.
		$$
		Importantly $U,\tilde{U}$ satisfy the same moment matching conditions
		we have imposed on $\tilde{W}_\tau$ and $W$. 
		Now by the fundamental theorem of calculus, we have for any symmetric matrix $W$,
			\begin{equation}
			\label{log via fundamental theorem}
			\log\left|\det(W + {\rm i}\eta_0)\right| = 
				\sum_{k=1}^N \log\left| x_k + {\rm i}\eta_0\right| = \log\left|\det(W + {\rm i})\right| - N\,
				 \Im  \int_{\eta_0}^1
				s_W\left({\rm i}\eta \right){\rm d}\eta.  
			\end{equation}
			From the central limit theorems for linear statistics of Wigner matrices on macroscopic scales \cite{LytPas2009}, 
			 $(\log\left|\det(W + {\rm i})\right|-\E(\log\left|\det(W + {\rm i})\right|))/\sqrt{\log N}$ 
			converges to $0$ in probability (the same result holds with $W$ replaced with $V$), and from Lemma \ref{expectation} (which clearly holds with $1$ in place of $\tau$), 
			$(\E(\log\left|\det(W + {\rm i})\right|)-\E(\log\left|\det(V + {\rm i})\right|))/\sqrt{\log N}\to 0$.
			Therefore (\ref{eqn:enough}) is equivalent to
				\begin{equation}
		\label{key expression}
			\lim_{N\to \infty} \mathbb{E} \left(\wt F\left(N\,
				 \Im  \int_{\eta_0}^1
				s_W\left({\rm i}\eta \right){\rm d}\eta\right) - \wt F\left(N\,
				 \Im  \int_{\eta_0}^1
				s_V\left({\rm i}\eta \right){\rm d}\eta\right)\right)= 0,
			\end{equation}
			where
			$$
			\wt F(x)=F\left(\frac{\E(\log\left|\det(W + {\rm i})\right|)+c_N-x}{\sqrt{\log N}}\right).
			$$

		We now expand $s_{W_1}$ and $s_{W_2}$ around $s_{Q}$, and then
		to Taylor expand $\wt F$.
		So let 
			\[ R =R(z)= \left(Q-z\right)^{-1} \text{ and } S=S(z)=\left(W_1 - z\right)^{-1}. \]
		By the resolvent expansion
			\[ 
			S = R - N^{-1/2}RUR + \hdots + N^{-2}(RU)^4R - N^{-5/2}(RU)^5 S,
			\]
		we can write
			\[N \int_{\eta_0}^1 s_{W_1}({\rm i}\eta){\rm d}\eta = \int_{\eta_0}^1\text{Tr}\left(S({\rm i}\eta)\right){\rm d}\eta = 
			\int_{\eta_0}^1 \text{Tr}\left(R({\rm i}\eta)\right){\rm d}\eta +
			\left( \sum_{m=1}^4 N^{-m/2}\hat{R}^{(m)}({\rm i}\eta) - N^{-5/2}\Omega\right)
				:= \hat{R} + \xi  \]
		where 
			\[ \hat{R}^{(m)} = (-1)^m \int_{\eta_0}^1 \text{Tr}\left((R({\rm i}\eta)U)^mR({\rm i}\eta)\right){\rm d}\eta 
			\quad \text{and} \quad
			\Omega = \int_{\eta_0}^1 \text{Tr} \left( (R({\rm i}\eta)U)^5S({\rm i}\eta) \right){\rm d}\eta.\]
		This gives us an expansion of $s_{W_1}$ around $s_{Q}$. 
		Now Taylor expand $\wt F(\hat{R} + \xi)$ as
			\begin{equation} \label{taylor expansion} \wt F\left(\hat{R}+\xi\right) = \wt F\left(\hat{R}\right) + 
				\wt F'\left(\hat{R}\right)\xi + \hdots + \wt F^{(5)}\left(\hat{R}+\xi'\right)\xi^5 = \sum_{m=0}^5 N^{-m/2}A^{(m)}
			\end{equation}
		where $0 < \xi' < \xi$, and we have introduced the notation $A^{(m)}$ in order to arrange terms according to
		powers of $N$. For example
			$$
				A^{(0)} = \wt F\left(\hat{R}\right),\ \ 
				A^{(1)} = \wt F'\left(\hat{R}\right)\hat{R}^{(1)}, \ \ 
				A^{(2)} = \wt F'\left(\hat{R}\right)\hat{R}^{(2)} + \wt F''\left(\hat{R}\right)\left(\hat{R}^{(1)}\right)^2.
			$$
		Making the same expansion for $W_2$, we record our two expansions as
			\[ \wt F\left(\hat{R} + \xi_i\right) = \sum_{m=0}^5 N^{-m/2} A^{(m)}_i, \quad i = 1, 2, \]
		with $\xi_i$ corresponding to $W_i$. With this notation, we have
			\begin{align*} 
				\mathbb{E}\left(\wt F\left(\hat{R} + \xi_1\right)\right) - \mathbb{E} 
				\left(\wt F\left(\hat{R} + \xi_2\right)\right)
				&=  \mathbb{E} \left(\sum_{m=0}^5 N^{-m/2}\left( A_1^{(m)} - A_2^{(m)} \right)\right) .
			\end{align*}
		Now only the first three moments of $U,\tilde{U}$ appear in the terms corresponding to $m=1, 2, 3$, so
		by the moment
		matching assumption (\ref{moment matching assumption 3}), all of these terms are all identically zero. 
		Next, consider $m = 4$. Every term with first, second, and third moments of
		$U$ and $\tilde{U}$ is again zero, and what remains is
			\[ \mathbb{E} \left(\wt F'(\hat{R}) \left( \hat{R}_1^{(4)} - \hat{R}_2^{(4)}\right)\right). \]  
		\noindent So we can discard $A^{(4)}$ if 
			\begin{equation}\label{4th moment}
			 \int_{\eta_0}^1 \left|\mathbb{E} \left( 
			 	{\rm Tr}\left( (RU)^4R \right) - {\rm Tr}\left( (R\tilde{U})^4R\right) \right)   \right| {\rm d}\eta 
			 \end{equation}
		is small. To see that this is in fact the case, we expand the traces, and apply Theorem \ref{local law}
		along with our fourth moment matching assumption (\ref{moment matching assumption 4}). 
		Specifically,
			\[ \text{Tr}\left( (RU)^4R\right) = \sum_j \left(\sum_{i_1, \hdots, i_8} 
				R_{ji_1}U_{i_1i_2}R_{i_2i_3}\hdots U_{i_7i_8}R_{i_8j}\right). \]
		Writing the corresponding Tr for $W_2$ and applying the moment matching assumption, we see that we can
		bound (\ref{4th moment}) by
			\[ {\rm O}(\tau) \int_{\eta_0}^1\sum_j \sum_{i_1, \hdots, i_8} \mathbb{E} \left(
				\left| R_{ji_1}R_{i_2i_3}R_{i_4i_5}R_{i_6i_7}R_{i_8j} \right| \right){\rm d}\eta. \]
		To bound the terms in the sum, we need to count the number of diagonal and off-diagonal terms in each product.
		To do this, let us say $U_{pq}, \tilde{U}_{pq}$ and $U_{qp}, \tilde{U}_{qp}$ are the only non-zero entries of 
		$U, \tilde{U}$. Then
		each of the sums over $i_1, \hdots, i_8$ are just sums over $p, q$, and when $j \notin \{p,q\}$, 
		$R_{ji_1}$ and $R_{i_8j}$ are certainly off-diagonal entries of $R$. This means we can apply Cauchy-Schwartz
		to write that for any $\gamma > 0$,
			\begin{align*}
				{\rm O}(\tau) \int_{\eta_0}^1\sum_{j \notin \{p, q\}} \sum_{i_1, \hdots, i_8} \mathbb{E} \left(
					\left| R_{ji_1}R_{i_2i_3}R_{i_4i_5}R_{i_6i_7}R_{i_8j} \right| \right){\rm d}\eta 
				= {\rm O}\left (\tau N^{1 + 2\gamma}  \int_{\eta_0}^1 \frac{1}{N\eta}{\rm d}\eta\right) = 
					{\rm O}\left (N^{2\gamma - \epsilon} \log(N)\right).
			\end{align*}
		Similarly, 
			\begin{align*}
				{\rm O}(\tau) \int_{\eta_0}^1\sum_{j \in \{p, q\}} \sum_{i_1, \hdots, i_8} \mathbb{E} \left(
					\left| R_{ji_1}R_{i_2i_3}R_{i_4i_5}R_{i_6i_7}R_{i_8j} \right| \right){\rm d}\eta 
				= {\rm O}\left (\tau N^{\epsilon /2} \right) = 
					{\rm O}\left (N^{- \epsilon/2} \right).
			\end{align*}
		Since $A^{(4)}$ has a pre-factor of $N^{-2}$ in (\ref{taylor expansion}), and the above holds
		for every choice of $\gamma > 0$, in our entire 
		entry swapping scheme starting from $V$ and ending with $W$, the corresponding error
		is ${\rm o}(1)$.\\
		
		\noindent Lastly we comment on the error term $A^{(5)}$. All terms in $A^{(5)}$ not involving 
		$\Omega$ can be dealt
		with as above. The only term involving $\Omega$ is $\wt F'(\hat{R})\Omega$, and to deal with this,
		we can expand the expression for $\Omega$ as above. We do not have any
		moment matching condition for the fifth moments of $U, \tilde{U}$, but (\ref{subgaussian}) means
		that their fifth
		moments are bounded which is enough for our purpose since $A^{(5)}$ has a pre-factor of $N^{-5/2}$ above. 
		\end{proof}	
		
	\subsection{Proof of Theorem \ref{main theorem}. }\label{sub:end}
	In this section we first prove Proposition \ref{variance} and, using Lemma 
	\ref{expectation}, we
	conclude the proof of Theorem \ref{main theorem}.

	\begin{proposition}\label{prop:variance} Recall $\tau=N^{-\e}$.
	\label{variance} There exist $\e_0,C$ such that for any fixed $0<\e<\e_0$, for large enough $N$, we have
	\[ {\rm Var}\left( \sum_k\log|x_k(0)+\ii\tau|\right) \leq C (1+\epsilon \log N).\]
\end{proposition}	
\begin{proof}
We outline two proofs, which are trivial extensions of existing linear statistics asymptotics on global scales, to the case of almost macroscopic scales. The tool for this extension is the rigidity estimate from \cite{ErdYauYin2012Rig}: for any $c,D>0$, there exists $N_0$ such that for any $N\geq N_0$ and $k\in\llbracket 1,N\rrbracket$ we have
\begin{equation}\label{eqn:rigidity}
\mathbb{P}\left(|x_k-\gamma_k|>N^{-\frac{2}{3}+c}\min(k,N+1-k)^{-\frac{1}{3}}\right)\leq N^{-D}.
\end{equation}

\noindent For the first proof, we use (\ref{eqn:rigidity}) to bound all the error terms in the proof of \cite[Theorem 3.6]{LytPas2009} (these error terms all depend on \cite[Theorem 3.5]{LytPas2009}, which can be improved via 
(\ref{eqn:rigidity}) to ${\rm Var}(u_N(t))\leq N^c(1+|t|)$ and ${\rm Var}\left(\mathcal{N}_N(\varphi)\right)\leq N^c\|\varphi\|_{\rm Lip}^2$). What we obtain is that if $\varphi$ (possibly depending on $N$) satisfies $\int|t|^{100}\hat \varphi(t)<N^{1/100}$, then $\sum\varphi(x_k)-\E(\sum\varphi(x_k))$ has limiting variance asymptotically equivalent to 
		\begin{equation}
\label{vwig} 
		V_{\rm Wig}[\varphi] =  \frac{1}{2\pi^2} \int_{(-2,2)^2}
			\left( \frac{\Delta \varphi}{\Delta \lambda}\right)^2 
			\frac{4 - \lambda_1\lambda_2}{\sqrt{4 - \lambda_1^2}
			\sqrt{1- \lambda_2^2}} \,
			{\rm d}\lambda_1\rd\lambda_2+ \frac{\kappa_4}{2\pi^2}
			\left( \int_{-2}^{2} \varphi(\mu) \frac{2 - \mu^2}{\sqrt{4 - \mu^2}}\,
			{\rm d}\mu\right)^2,
		\end{equation}
where $\Delta \varphi = \varphi\left(\lambda_1\right) - \varphi\left(\lambda_2\right)$, $\Delta \lambda = \lambda_1 - \lambda_2$, $\mu_4 = \mathbb{E} \left(W_{jk}^4\right)$,
	$\kappa_4 = \mu_4 - 3$ is the fourth cumulant of the off-diagonal entries of $W$.
We choose $\varphi(x)=\varphi_N(x)=\frac{1}{2}\log(x^2+\tau^2)\chi(x)$ with $\chi$ fixed, smooth, compactly supported, equal to 1 on $(-3,3)$. Note that for $\e_0$ small enough, we have $\int|t|^{100}\hat \varphi(t)<N^{1/100}$. Then by (\ref{eqn:rigidity}) and (\ref{vwig}), 
$$
V_{\rm Wig}[\log|\cdot-\ii\tau|]\sim V_{\rm Wig}[\varphi]\leq C \iint \left(\frac{\Delta\varphi}{\Delta\lambda}\right)^2
	\, \rd \lambda_1 \rd \lambda_2
	= C \int |\xi|\, \left| \hat{\varphi}(\xi) \right|^2{\rm d}\xi,
$$
and the above right hand side can be bounded as follows.
	We have
		\begin{align*}
		 \left| \hat{\varphi}_N(\xi) \right| = \left| \frac{1}{2\pi} \int_{\mathbb{R}}
		 	 \varphi_N(x)e^{-i\xi x}\, \rd x\right|
		 &\leq 
		 C \left| \int_{-5}^{5} \frac{x}{x^2 + \tau^2} \frac{e^{-i\xi x}}{i\xi}\, \rd x\right| 
		 =C\left| \frac{1}{\xi} \int_0^{5/\tau} \frac{x}{x^2 + 1} \sin(x \xi \tau) \, \rd x \right|.
		\end{align*}
	For $0 < \xi < 5$, the inequality $|\sin x| < x$ shows $\left| \hat{\varphi}_N(\xi) \right| = {\rm O}(1)$, and when
	$\xi > 5/\tau$, integration by parts shows $\left|\hat{\varphi}_N(\xi)\right| = {\rm O}\left (\frac{1}{\xi^2\tau}\right)$. 
	When $5 < \xi < 5/\tau$, first note
		\begin{align*}
		\int_0^{\frac{5}{\tau}} \sin\left(\xi \tau x\right) \frac{x}{x^2 + 1} \, {\rm d}x &= C + 
			\int_1^{\frac{5}{\tau}} \frac{\sin\left(\xi \tau x\right)}{x} \,{\rm d}x 
			=  C + \int_{\xi \tau}^{1} \frac{\sin y}{y} \,{\rm d}y + \int_1^{5\xi} \frac{\sin y}{y}\,{\rm d}y.
		\end{align*}
	Using $|\sin y| < |y|$, we see that the first term is ${\rm O}(1)$, and integrating by parts, we see
	that the second term is ${\rm O}(1)$ as well. This means 
		\[ \int |\xi| \left| \hat{\varphi}_N(\xi) \right|^2{\rm d}\xi \leq C +C 
		\int_{5}^{\frac{5}{\tau}} \frac{1}{\xi} \,{\rm d}\xi = {\rm O}\left (1+|\log \tau|\right), 
		\]
	which concludes the proof.\\
	
\noindent The  second proof is similar but more direct. Theorem 3 in \cite{KhoKhoPas} implies that for $z_1=\ii\eta_1,z_2=\ii\eta_2$ at macroscopic distance from the real axis, and $\eta_1=\im z_1>0,\eta_2=\im z_2<0$, we have
	$$
	\left|{\rm Cov}\left(\sum_k\frac{1}{z_1-x_k},\sum_k\frac{1}{z_2-x_k}\right)\right|\leq\frac{C}{(\eta_1-\eta_2)^2}+f(z_1,z_2)+\OO(N^{-1/2}),
	$$
	where $f$ is a function uniformly bounded on any compact subset of $\mathbb{C}^2$.
	Using (\ref{eqn:rigidity}), one easily obtains that the formula above holds uniformly with $|\im z_1|,|\im z_2|>N^{-1/10}$, and the deteriorated error term $\OO(N^{-1/10})$, for example.	
Note that
$$
\log\left|\det(W + {\rm i}\eta)\right| = \log\left|\det(W + {\rm i})\right| - N\,
				 \Im  \int_{\eta}^1
				s_W\left({\rm i}x \right){\rm d}x.  
$$
and $\log\left|\det(W + {\rm i})\right|$ has fluctuations of order 1 due to the above macroscopic central limit theorems. 
For for $\eta>N^{-1/10}$, the variance of the above integral can be bounded by 
$
\iint_{[\eta,1]^2}\frac{1}{|\eta_1+\eta_2|^2}\,\rd \eta_1\rd \eta_2\leq C|\log \eta|,
$
which concludes the proof.
	\end{proof}

\noindent From (\ref{GOE CLT}) and Proposition \ref{smoothing}, for some explicit deterministic  $c_N$ we have
\begin{equation}\label{eqn:1}
\frac{ 			\sum_{k =1}^N \log \left|x_k(\tau) + {\rm i} \eta_0\right|+c_N}{\sqrt{\log N}}\to\mathscr{N}(0,1),
\end{equation}
and Proposition \ref{prop:advection} implies that
$$
\frac{ 			\sum_{k =1}^N \log \left|y_k(\tau) + {\rm i} \eta_0\right|+c_N}{\sqrt{\log N}}
+\frac{ 			\sum_{k =1}^N \log \left|x_k(0) + z_\tau\right|-\sum_{k =1}^N \log \left|y_k(0) + z_\tau\right|}{\sqrt{\log N}}
\to\mathscr{N}(0,1).
$$

	\noindent Lemma \ref{expectation} and Proposition \ref{variance} show that the second term above, call it $X$, satisfies $\E(X^2)<C\e$, for some universal $C$. Thus for any fixed smooth and compactly supported function $F$,
	\begin{align*}
	\E \left(F\left(\frac{ 			\sum_{k =1}^N \log \left|y_k(\tau) + {\rm i} \eta_0\right|+c_N}{\sqrt{\log N}}\right)\right)&=\E \left(F\left(\frac{ 			\sum_{k =1}^N \log \left|x_k(\tau) + {\rm i} \eta_0\right|+c_N}{\sqrt{\log N}}+X\right)\right)
	+\OO\left(\|F\|_{\rm Lip}(\E\left(X^2\right))^{1/2}\right)\\
	&=\E \left(F(\mathscr{N}(0,1))\right)+\oo(1)+\OO\left(\e^{1/2}\right).
	\end{align*}
With Theorem \ref{4 moment matching theorem}, the above equation implies
$$
\E\left( F\left(\frac{\log|\det(W+\ii\eta_0)|+c_N}{\sqrt{\log N}}\right)\right)=\E \left(F(\mathscr{N}(0,1))\right)+\oo(1)+\OO\left(\e^{1/2}\right),
$$		
and by Proposition \ref{smoothing}, we obtain
\begin{equation}\label{eqn:2}
\E \left(F\left(\frac{\log|\det W|+\frac{N}{2}}{\sqrt{\log N}}\right)\right)=\E \left(F(\mathscr{N}(0,1))\right)+\oo(1)+
\OO\left(\e^{1/2}\right).
\end{equation}
Since $\e$ is arbitrarily small, this concludes the proof.
\nc
\setcounter{equation}{0}
\setcounter{theorem}{0}
\renewcommand{\theequation}{A.\arabic{equation}}
\renewcommand{\thetheorem}{A.\arabic{theorem}}
\appendix
\setcounter{secnumdepth}{0}
\section[Appendix A\ \ \ Expectation of Regularized Determinants]
{Appendix A:\ \ \ Expectation of Regularized Determinants}

\noindent
We prove the following result, which we use both in the proof of Theorem \ref{4 moment matching theorem},
and to conclude the proof of Theorem \ref{main theorem}. 

\begin{lemma}\label{expectation}
	Recall the notation $\tau = N^{-\epsilon}$, and let $\{x_k\}_{k=1}^N$, $\{y_k\}_{k=1}^N$ denote the
	eigenvalues of two Wigner matrices, $W_1$ and $W_2$. Then
		\[ \mathbb{E}\left( \sum_k \log \left|x_k + {\rm i}\tau\right| - \sum_k \log\left|y_k + {\rm i}\tau\right|\right) = \OO(1).\]
	\end{lemma}
	\begin{proof}
		By the fundamental theorem of calculus, we can write
		\begin{equation}
		\label{ftc log det}
			\sum_k \log \left|x_k + {\rm i}\tau \right| = \sum_{k=1}^N\log \left|x_k + {\rm i} N^{\delta}\right| + 
			N \int_\tau^{N^{\delta}} \Im \left(s_{W_1}({\rm i} \eta)\right){\rm d}\eta
		\end{equation}
	with $s_W$ as in ({\ref{def:stransform}}), and $\delta > 0$. 
	Writing the same expression for $W_2$ and taking the difference, we first note that by 
	(\ref{eqn:rigidity}),
	we have that for any $\gamma > 0$,
	         \begin{equation}\label{eqn:boundexp}
		  \mathbb{E}\left( \left|\sum_{k=1}^N  \left(\log\left|x_k + {\rm i} N^{\delta}\right| - 
			\log\left|y_k + {\rm i} N^{\delta}\right|\right)\right| \right)
			\leq \mathbb{E}\left(N^{-\delta} \sum_{i=1}^N \left| x_k - y_k\right|\right) = 
			{\rm O}\left( N^{\gamma - \delta} \right).
		\end{equation}
	Therefore, we only need to bound
		\begin{equation}
			\label{main term expectation}
			\Im  \left( N \int_\tau^{N^{ \delta }} \mathbb{E}\left(s_{W_1}(\ii \eta) - s_{W_2}(\ii \eta)\right) {\rm d}\eta \right).
		\end{equation}
	Let $z = E + {\rm i}\eta$ be in $S\left( \frac{1}{100} \right)$ (as defined in Theorem
	\ref{local law}), and define
		\[ f(z) = N\left(s_{W_1}(z) - s_{W_2}(z)\right). \]
	
	\noindent We will first estimate $\E\left(f(z)\right)$ for $\tau < \eta < 5$, where we can use
	Theorem \ref{local law}. Then we will
	use complex analysis to extend this estimate to $5 < \eta < N^{\delta}$. \\
	
	\noindent Let $\tau < \eta < 5$.
	Following the notation of \cite{ErdYauYin2012Rig}, let $W$ be a Wigner matrix and let
		\begin{equation*}
			v_i = G_{ii} - m_{sc}, \quad [v] = \frac{1}{N}\sum_{i=1}^N v_i,
			\quad G(z) = (W-z)^{-1}, 
		\end{equation*}
	We will use the notation
		$W^{(i)}$ 
	to
	denote the $(N -1) \times (N-1)$ matrix obtained by removing the $i^{\text{th}}$ row and column from $W$,
	and $w_i$ to denote the $i^{\text{th}}$ column of $W^{(i)}$ without $W_{ii}$. 
	We will also denote the eigenvalues of $W$
	by $ \lambda_1 < \lambda_2 < \dots \lambda_N$.
	Let
	$G^{(i)} = \left(W^{(i)} - z\right)^{-1}$.
	Applying the Schur complement formula to $W$ (see Lemma 4.1 in \cite{ErdYauYin2012Bulk}), we have 
		\begin{equation}
		\label{schur complement}
			v_i + m_{sc} =  \left( -z - m_{sc} + W_{ii} - [v] + \frac{1}{N}\sum_{j\neq i} 
				\frac{G_{ij}G_{ji}}{G_{ii}} - Z_i \right)^{-1} = 
				\left( -z -m_{sc} - \left([v]- \Gamma_i \right)\right)^{-1} 
		\end{equation}
	where
		\begin{equation*}
			Z_i = (1 - \mathbb{E}_i)(w_i, G^{(i)}w_i), \quad \mathbb{E}_i(X) = \mathbb{E}\left(X | W^{(i)}\right),
			\quad
			\Gamma_i =  \frac{1}{N}\sum_{j\neq i} \frac{G_{ij}G_{ji}}{G_{ii}} - Z_i + W_{ii}.
		\end{equation*}
	By Theorem \ref{local law}, 
		\begin{equation}
		\label{crude estimate for gammai}
			\left|\Gamma_i - \left[v\right] \right| = \OO_\prec\left( \frac{1}{N^\frac{1}{2}\eta^\frac{1}{2}}\right), 
		\end{equation}
	so we can expand (\ref{schur complement}) around $-z-m_{sc}$. Using 
	(\ref{self consistent m}), we find
		\begin{align*}
			v_i &= m_{sc}^2\left( [v] - \Gamma_i \right) + m_{sc}^3 \left( [v] - \Gamma_i \right)^2
			+ {\rm O}\left ( \left( [v] - \Gamma_i \right)^3 \right) \\
			&= m_{sc}^2\left([v] - W_{ii} -\frac{1}{N}\sum_{j\neq i}  \frac{G_{ij}G_{ji}}{G_{ii}}
				+ Z_i\right) + m_{sc}^3 \left([v] - \Gamma_i \right)^2
			+ {\rm O}\left ( \left( [v] - \Gamma_i \right)^3 \right),
		\end{align*}
	and summing over $i$ and taking expectation, we have
		\begin{equation}\label{vi} \mathbb{E} \left((1-m_{sc}^2) \sum_i v_i \right)
		= \mathbb{E} \left( -\frac{m_{sc}^2}{N} \sum_{i=1}^N \sum_{j\neq i}^N \frac{G_{ij}G_{ji}}{G_{ii}} + 
			m_{sc}^3 \sum_i\left([v]- \Gamma_i \right)^2
			+ \sum_i{\rm O}\left ( \left( [v] - \Gamma_i \right)^3 \right)\right),
		\end{equation}
	since the expectations of $W_{ii}$ and $Z_i$ are both zero. We now use this expansion to estimate $\E(f(z))$.
	Since we $\tau < \eta < 5$, we have by 
	Theorem \ref{local law} that
		\begin{align}
		\label{first term of expansion}
		\frac{m_{sc}^2}{N} \sum_i \sum_{j\neq i} \frac{G_{ij}G_{ji}}{G_{ii}} &=
			\frac{m_{sc}}{N}\left(
				\sum_{i,j =1}^N G_{ij}G_{ji} - \sum_{i=1}^N \left(G_{ii}\right)^2
			\right)
			+ {\rm O}_\prec
			\left( \frac{1}{N^{\frac{1}{2}}\eta^{\frac{1}{2}}}\right) 
			\frac{m_{sc}}{N} \sum_{i}\sum_{j \neq i} \left|G_{ij}G_{ji}\right|.
		\end{align}
	Now observe that 
	\begin{align*}
		\frac{m_{sc}}{N} \sum_{i,j}G_{ij}G_{ji} &= \frac{m_{sc}}{N} \text{Tr} \left(G^2\right)
		= \frac{m_{sc}}{N} \sum_{k=1}^N \frac{1}{\left( \lambda_k - z \right)^2},
		 \end{align*}
	and
		\[ \frac{1}{N}\sum_{k=1}^N \frac{1}{\left( x_k - z \right)^2} - 
			\frac{1}{N}\sum_{k=1}^N\frac{1}{\left(y_k - z \right)^2} =  
			s'_{W_1}(z) - s'_{W_2}(z).
		\]
	Choosing $\mathcal{C}(z) = \left\{ w \,:\, |w-z| = \frac{\eta}{2} \right\}$, we have 
		\begin{align}
			\label{complex derivative}
			\left| s'_{W_1}(z) - s'_{W_2}(z) \right|
			\leq 
			\frac{1}{2\pi} \int_{\mathcal{C}(z)} \frac{\left| s_{W_1}(z) - s_{W_2}(z)  \right|}{(\zeta-z)^2} \,{\rm d}\zeta
			= {\rm O}_\prec\left( \frac{1}{N\eta^2} \right)
		\end{align}
	by Theorem \ref{local law}.
	Again applying Theorem \ref{local law}, we have 
		\[ 
		\frac{m_{sc}}{N}\sum_{i=1}^N \left(G_{ii}\right)^2
		= \frac{m_{sc}}{N}\sum_{i=1}^N \left(v_i + m_{sc}\right)^2 = m_{sc}^3 + {\rm O}_\prec
		\left( \frac{1}{N\eta} \right),
		\text{ and }
		 \sum_{i \neq j} \left|G_{ij}G_{ji}\right| = {\rm O}_\prec\left(\frac{1}{\eta}\right). 
		\]
	Putting together these estimates we have
		\begin{align*}
			&\mathbb{E}\left(\int_{\tau}^5 \sum_{i=1}^N\sum_{j\neq i}^N \frac{m_{sc}^2}{N\left(1 - m_{sc}^2\right)}
			\left( \frac{\left(G_1\right)_{ij} \left(G_1\right)_{ji}}{\left(G_1\right)_{ii}} - 
			\frac{\left(G_2\right)_{ij}\left(G_2\right)_{ji}}{\left(G_2\right)_{ii}} \right){\rm d}\eta
			\right)
			= \mathbb{E} \left( \int_\tau^5 
				{\rm O}_\prec\left( \frac{1}{N^{\frac{1}{2}}\eta} \right) {\rm d}\eta \right) = {\rm o}(1).
		\end{align*}
	\noindent Next, consider
		\begin{equation} 
		\label{2nd term}
		m_{sc}^3 \sum_{i=1}^N \left([v]- \Gamma_i \right)^2 = m_{sc}^3
		\sum_{i=1}^N \left([v]^2 - 2[v]\Gamma_i + \Gamma_i^2\right).
		\end{equation}
	By Theorem \ref{local law}, $[v] = {\rm O}_\prec\left ( \frac{1}{N\eta}\right)$,
	so summing over $i$ and integrating with respect to $\eta$, we find
		\[ \mathbb{E} \left(\int_\tau^5 \sum_i \frac{m_{sc}^3}{1- m^2_{sc}} [v]^2 {\rm d}\eta\right) = 
		\mathbb{E} \left( \int_\tau^5 {\rm O}_\prec\left( \frac{1}{N\eta^{\frac{5}{2}}} \right)\right) = {\rm O}
		\left( \frac{N^{ \frac{3\epsilon}{2} +\gamma }}{N} \right)\]
	for any $\gamma > 0$.
	Next, we estimate $\mathbb{E} \left(m_{sc}^3\sum_i\Gamma_i^2\right)$. Expanding $\Gamma_i^2$, we have
		\begin{equation}
		\label{Gammai}
		 \Gamma_i^2 =  W_{ii}^2 + \left( \frac{1}{N} \sum_{j \neq i} \frac{G_{ij}G_{ji}}{G_{ii}} \right)^2
		+ Z_i^2 + 2\left( \frac{W_{ii}}{N} \sum_{j \neq i} \frac{G_{ij}G_{ji}}{G_{ii}} -
		W_{ii}Z_i - \frac{Z_i}{N} \sum_{j \neq i} \frac{G_{ij}G_{ji}}{G_{ii}} \right).
		\end{equation}
	By definition, we have $\mathbb{E}\left(W_{ii}^2\right) = \frac{1}{N}$. Therefore
	$\mathbb{E}\left(\left(W_1\right)_{ii}^2 - \left(W_2\right)_{ii}^2\right) = 0$, and by Theorem \ref{local law}, we have
		\[ \sum_{i=1}^N m_{sc}^3 \left( \frac{1}{N} \sum_{j \neq i} \frac{G_{ij}G_{ji}}{G_{ii}} \right)^2 =
			{\rm O}_\prec\left( \frac{1}{N\eta^2} \right).\]
	Next, we examine $\mathbb{E}\left(\sum_{i=1}^N Z_i^2\right)$. Note that by the independence of 
	$w_i(l)$ and $w_i(k)$ and the independence of $w_i$ and $G^{(i)}$, we have
		\[ \mathbb{E}_i \left( \left< w_i, G^{(i)}w_i\right>\right) = \mathbb{E}_i \left( \sum_{k, l} G_{kl}^{(i)} w_i(l)\overline{w_i(k)}\right) =
		\mathbb{E}_i\left(\sum_{k=1}^N G_{kk}^{(i)} \overline{w_i^2(k)} \right)
			= \frac{1}{N} \text{Tr}\left(G^{(i)}\right). \]
	Therefore,
		\begin{equation}
		\label{zi2}
		 \mathbb{E} \left(\sum_{i=1}^N Z_i^2\right) = 
		 \sum_{i=1}^N
		 \mathbb{E}_{W^{(i)}} \left( \mathbb{E}_i \left( \left(\left<w_i, G^{(i)}w_i\right>^2\right) -
			\left( \frac{1}{N} \text{Tr}\left(G^{(i)}\right) \right)^2 \right) \right).
		\end{equation}
	Expanding the first term on the left hand side above, we have
		\begin{equation}
		\label{cases} \mathbb{E}_i   \left(\left<w_i, G^{(i)}w_i\right>^2\right) = 
			\mathbb{E}_i \left( \sum_{k, l, k', l'} G_{kl}^{(i)}w_i(l)\overline{w_i(k)} G_{k'l'}^{(i)}w_i(l')\overline{w_i(k')}\right).
		\end{equation}
	The only terms which contribute to this sum are
	those for which at least two pairs of the indices amongst $k, k', l, l'$ coincide. Consider first the case
	$k = l$, $k' = l'$, $k \neq k'$. The contribution of these terms to the above sum is
		\[ \mathbb{E}_i\left(\sum_{k\neq l} G_{kk}^{(i)} G_{ll}^{(i)} \left|w_i(k)\right|^2\left|w_i(l)\right|^2\right) = 
		\left(\frac{1}{N} \text{Tr}\left(G^{(i)}\right)\right)^2
			- \frac{1}{N^2} \sum_{k=1}^N \left(G^{(i)}_{kk}\right)^2.\]
	The first term on the right hand side here cancels the second term on the right hand side of (\ref{zi2}). 
	For the second term, by Theorem \ref{local law}, we have
		\begin{equation}
		\label{blah}
		 \frac{1}{N^2} \sum_{i=1}^N\sum_{k=1}^N \left(\left(G_1^{(i)}\right)_{kk}^2 - 
			\left(G_2^{(i)}\right)_{kk}^2\right) = {\rm O}_\prec\left( \frac{1}{N^{\frac{1}{2}}\eta^{\frac{1}{2}}} \right). 
		\end{equation}  
	Next consider the case where $k = k'$, $l = l'$, $k \neq l$. 
	We consider separately the case where $W$ has real entries,
	and the case where $W$ has complex entries. In the first case, we can assume that the eigenvectors
	of $W$ have real entries. Therefore, by the spectral decomposition of $G$, we have
		\[ \frac{1}{N^2}\sum_{i=1}^N \sum_{k \neq l} \left(G^{(i)}_{kl}\right)^2 =
		\frac{1}{N^2}\sum_{i=1}^N\left( \sum_{k, l} \left(G^{(i)}_{kl}\right)^2 -
		\sum_{k\neq i} \left(G^{(i)}_{kk}\right)^2  \right) =
		 \frac{1}{N^2}\sum_{i=1}^N \sum_{k\neq i}\left( \frac{1}{\left( \lambda_k^{(i)} - z \right)^2} -
			\left(G^{(i)}_{kk}\right)^2   \right). 
		\]
	Using (\ref{complex derivative}) and (\ref{blah}), this gives us
		\[  \frac{1}{N^2}\sum_{i=1}^N \sum_{k \neq l}  \left(\left(G_1^{(i)}\right)_{kl}^2 -
			\left(G_2^{(i)}\right)_{kl}^2\right)  = 
		\OO_\prec\left( \frac{1}{N^{\frac{1}{2}} \eta^2} \right). \]
	If instead $W$ has complex entries, this term is identically zero. Indeed the corresponding expression becomes
		\[ \sum_{i=1}^N \sum_{k \neq l} \left(G^{(i)}_{kl}\right)^2 
		\mathbb{E}_i\left( \left( \overline{w_i(k)} \right)^2\left( w_i(l) \right)^2 \right),\]
	and because we have assumed that
	that for $i \neq j$,
	$W_{ij}$ is of the form $x + {\rm i}y$ where $\mathbb{E}(x) = \mathbb{E}(y) = 0$ and $\mathbb{E}\left(x^2\right)
		= \mathbb{E}\left(y^2\right)$, we have $\mathbb{E}\left(W_{ij}\right)^2 = 0$.
	There remain two cases to consider. Suppose $k' = l$, $l' = k$, $k\neq l$.
	Then 
		\[ \sum_{i=1}^N\mathbb{E}_i \left( \sum_{k \neq l} G^{(i)}_{kl}G^{(i)}_{lk}\left| w_i(k)\right|^2
		\left|w_i(l)\right|^2\right) = \sum_{i} \frac{1}{N^2}\left( 
		\sum_{k, l} G_{kl}^{(i)}G_{lk}^{(i)} - \sum_{k=1}^N \left(G_{kk}^{(i)}\right)^2
		\right), \]
	and we may estimate the difference of this expression at $G_1$ and $G_2$ as we did the first term on the right hand 	side of (\ref{first term of expansion}).
	Lastly, we consider the case $k = k' = l = l'$. 	
	By Definition \ref{def wigner} and Theorem \ref{local law}, there exists a constant $C$ such that
		\begin{equation}
		\label{q}
		 \sum_{i=1}^N \mathbb{E}_i \left(\sum_{k=1}^N \left(G^{(i)}_{kk}\right)^2\left|w_i(k)\right|^4 \right)
		   = C m^2_{sc}(z) + {\rm O}_\prec \left( \frac{1}{N^{\frac{1}{2}}\eta^{\frac{1}{2}}} \right).
		\end{equation}
	Therefore
		\[
		\sum_{i=1}^N \mathbb{E}_i \left(\sum_{k=1}^N \left( G_1^{(i)} \right)_{kk}^2 \left|w_i^{(1)}(k)\right|^4 
		- \left( G_2^{(i)} \right)_{kk}^2 \left|w_i^{(2)}(k)\right|^4
		\right)
		   = Cm_{sc}^2(z) + {\rm O}_\prec \left( \frac{1}{N^{\frac{1}{2}}\eta^{\frac{1}{2}}} \right).
		\]
	In summary,
		\begin{equation} 
		\label{z summary}
		\mathbb{E}\left(\sum_{i=1}^N \left[\left(Z_1\right)_i^2 -
		\left(Z_2\right)_i^2\right]\right) = \OO\left( 1 \right). \end{equation}
Returning to (\ref{Gammai}), by Theorem \ref{local law} we have 
		\[ \mathbb{E} \left(\sum_{i=1}^N \frac{W_{ii}}{N} \sum_{j\neq i} \frac{G_{ij}G_{ji}}{G_{ii}}\right) 
		\leq \sum_{i=1}^N \left(\left(\mathbb{E}\left(W_{ii}^2\right)\right)^\frac{1}{2}
			\left( \mathbb{E} \left( \frac{1}{N} \sum_{j\neq i} \frac{G_{ij}G_{ji}}{G_{ii}} \right)^2  \right)^{\frac{1}{2}}\right)
			= {\rm O}\left( \frac{N^\gamma}{N^{\frac{1}{2}}\eta} \right)\]
	for any $\gamma > 0$.
	We also have that $\mathbb{E}\left( W_{ii}Z_i \right)= 0$.
	To bound the remaining term in (\ref{Gammai}), we 
	first note that using the same argument as we did to prove
	(\ref{z summary}), we have
			\begin{equation}
				\label{|zi|2}
		\mathbb{E}\left(\left|Z_i\right|^2\right) = {\rm O}\left( \frac{1}{N\eta}  \right).
			\end{equation}
	Applying Theorem \ref{local law}, we therefore conclude that
		\[ \mathbb{E}\left(\left|\sum_{i=1}^N \frac{Z_i}{N} \sum_{j\neq i} \frac{G_{ij}G_{ji}}{G_{ii}}\right| \right)= 
		{\rm O}\left( \frac{N^\gamma}{N\eta^2} \right), \]
	for any $\gamma > 0$. Putting together all of our estimates concerning
	(\ref{Gammai}), we have
		\begin{equation}
		\label{gamma summary}
			\mathbb{E} \left(\int_{\tau}^5 \sum_{k=1}^N
			\left( \frac{m_{sc}^3}{1 - m_{sc}^2} \Gamma_{k}^2\right) {\rm d}\eta\right) = \OO(1),
		\end{equation}
		where we used $\frac{m_{sc}^3}{1 - m_{sc}^2}=\OO(1)$.
	Returning to (\ref{2nd term}), by Cauchy-Schwarz and Theorem \ref{local law}
	we have that for any $\gamma > 0$ 
		\[ \mathbb{E} \left(\sum_{i=1}^N m_{sc}^3 [v] \Gamma_i \right)= 
		{\rm O}\left( \frac{N^\gamma}{N^{\frac{1}{2}} \eta^{\frac{3}{2}}} \right). \]
	In total, we have
		\begin{equation} 
			\label{second term summary}
			\mathbb{E} \left(\int_{\tau}^5 \left(\frac{m_{sc}^3}{1 - m_{sc}^2}\right) \sum_{i=1}^N \left( 
			[v]^2 - 2[v]\Gamma_i + \Gamma_i^2\right)
				{\rm d}\eta\right) = \OO\left( 1\right).
		\end{equation} 
	Finally, we have
		\[
			\int_{\tau}^5 \sum_i \left| [v] - \Gamma_i \right|^3 \,\rd \eta = \oo(1)
		\]
	using (\ref{crude estimate for gammai}). \\

	\noindent In summary, we have proved that for $z = \left(E + {\rm i}\eta\right)\in S\left( \frac{1}{100}\right)$, 
	and any $\gamma > 0$,
		\begin{equation}
		\label{estimate up to 1}
		 \mathbb{E} \left(f(z)\right) = \frac{C m^5_{sc}(z)}{1 - m^2_{sc}(z)} + \OO\left( 
			\frac{N^\gamma}{N^{\frac{1}{2}}\eta^{ \frac{5}{2}  }}\right).
		\end{equation}
	In particular, this means that 
		\[ \int_\tau^5 \mathbb{E}\left(f\left( {\rm i}\eta\right)\right){\rm d}\eta = 
			{\rm O}(1). \]
	To complete the proof of this lemma, we need to estimate 
		$ \int_5^{N^\delta} \mathbb{E}\left(f\left( {\rm i}\eta\right)\right){\rm d}\eta$. 
	Let 
		\[ q(z) =\mathbb{E}\left(f(z)\right), \quad \tilde{q}(z) = q\left(\frac{1}{z}\right).\]
	The function $q$ is clearly bounded as $|z|\to\infty$, so $\tilde{q}$ is bounded at $0$, which 
	by Riemann's theorem is therefore a removable singularity. 
	By (\ref{eqn:rigidity}), this means
		\[ \mathbb{P}\left( \tilde{q}(z) \text{ is analytic in } \mathbb{C} \backslash\left\{ 
		\left(-\infty, -\frac{1}{3}\right)\cup \left(\frac{1}{3}, \infty\right) 
		\right\}\right) \geq 1 - N^{-D},
		\]
	and so with overwhelming probability, we can write
		\begin{equation}
		\label{CIF}
			q(z) = \tilde{q}(w) = \frac{1}{2\pi{\rm i}}\int_{\mathcal{C}_\Gamma} \frac{\tilde{q}(\xi)}{\xi - w}\,{\rm d}\xi 
			= - \frac{1}{2\pi\ii}\int_{C_\gamma} \frac{q(\xi)}{\xi - w \xi} \, \rd \xi
		\end{equation}
	where $w = \frac{1}{z}$ and we choose
	$\mathcal{C}_\gamma = \left\{ x + {\rm i}y\,:\, \left|x\right| = 4,\, \left|y\right| = 4 \right\}$
	so that $w$ is inside $C_\Gamma$, and $\tilde{q}$ is analytic there. 
	Now
	we can estimate the right hand side using (\ref{estimate up to 1})
	and (\ref{eqn:rigidity}). Since $\Im(z) > 5$, we have $\sup_{\xi \in C_\gamma} \frac{1}{|\xi - w\xi|} = \OO(1)$.
	Furthermore, for
	$z \in \left[ 4 - \ii \tau, 4 + \ii \tau \right]$, by (\ref{eqn:rigidity}) we have
		\[ \left|f(z)\right| = \left|\sum_{k=1}^N\left( \frac{1}{x_k - z} - \frac{1}{y_k - z}\right)\right|
			= {\rm O}_\prec \left(1\right). 
		\] 
	Therefore, using (\ref{estimate up to 1}), when $\left| \Im (z) \right| > 5$, for any $\gamma > 0$, we have,
		\[ |q(z)| \leq
		\sup_{\xi \in C_\gamma} \frac{1}{|\xi - w\xi|}
		\OO\left(
		\int_{-4}^{4}
			\frac{N^\gamma}{N^\frac{1}{2}} {\rm d}x +
			\int_{\tau}^4
			\frac{N^\gamma}{N^\frac{1}{2}y^{ \frac{5}{2}  }}{\rm d}y
			+ \int_0^{\tau}
			N^\gamma \,{\rm d}y
		\right) = {\rm O}\left( N^{\gamma - \e} 
	\right),
		\]
	and so
		\begin{equation}
		\label{final} \int_5^{N^\delta}\left| \mathbb{E}\left(f(z)\right)\right|{\rm d}\eta = 
			\int_5^{N^\delta} \left( \frac{C\cdot m^5_{sc}(z)}{1 - m^2_{sc}(z)} + \OO\left( 
			N^{\gamma - \e}\right)\right) {\rm d}\eta = 
			\OO\left(1 \right) + \OO\left( N^{\gamma - \e + \delta} \right).
		\end{equation}
	This completes the proof of Lemma \ref{expectation}.
	 \end{proof}

\setcounter{equation}{0}
\setcounter{theorem}{0}
\renewcommand{\theequation}{B.\arabic{equation}}
\renewcommand{\thetheorem}{B.\arabic{theorem}}
\appendix
\setcounter{secnumdepth}{0}
\section[Appendix B\ \ \ Fluctuations of Individual Eigenvalues]
{Appendix B:\ \ \ Fluctuations of Individual Eigenvalues}

In this appendix, we prove
Theorem \ref{fluctuations of individual eigenvalues}. 
The main observation is that the determinant corresponds to linear statistics for the function  ${\Re}\log$, 
while individual eigenvalue fluctuations correspond to the central limit theorem for ${\Im}\log$. We build on this parallel below. The main step is Proposition \ref{eigenvalue counting}, which considers only the case $m = 1$, the proof for the multidimensional central limit theorem being strictly similar.\\

\noindent In analogy with (\ref{log via fundamental theorem}), for any $\eta \geq 0$, define
	\begin{equation}
	\label{def: im log}
		\Im \log \left( E + \ii \eta \right) = \Im \log \left(E + \ii \infty \right)  
		- \int_{\eta}^\infty \Re\left( \frac{1}{E-\ii u} \right) \, \rd u ,
	\end{equation}
with the convention that $\Im \log\left(E + \ii \infty\right) = \frac{\pi}{2}$. 
Then we can write
	\begin{equation}
	\label{def: im log arctan}
		\Im \log \left( E + \ii \eta \right) = \frac{\pi}{2} - \arctan\left( \frac{E}{\eta} \right),
	\end{equation}
and as $\eta \to 0^+$, we have
	\[
		\Im \log(E) = \begin{cases}
			0 & E > 0 \\
			\pi & E < 0. 
		\end{cases}
	\]
\begin{proposition}
\label{eigenvalue counting}
	Let $W$ be a real Wigner matrix satisfying (\ref{subgaussian}). Then with $\Im \log \det (W -E)$ defined
	as
	\[
		\Im \log\left( \det (W-E) \right) = \sum_{k=1}^N \Im \log\left( \lambda_k - E \right),
	\]
	we have
	\begin{equation}
		\frac{\frac{1}{\pi}
			\Im \log\left( \det (W-E) \right) - N\int_{-\infty}^E \rho_{sc}(x)\, \rd x}{\frac{1}{\pi}\sqrt{\log N}} \to \mathscr{N}(0, 1).  
	\end{equation}
	If $W$ is a complex Wigner matrix satisfying (\ref{subgaussian}), then
	\begin{equation}
		\frac{\frac{1}{\pi}
			\Im \log\left( \det (W-E) \right) - N\int_{-\infty}^E \rho_{sc}(x)\, \rd x}{
			\frac{1}{\pi}\sqrt{ \frac{1}{2}\log N}} \to \mathscr{N}(0, 1). 
	\end{equation}
\end{proposition}	

\noindent Before proving Proposition \ref{eigenvalue counting}, 
we prove Lemma \ref{eigenvalue counting to individual fluctuations} which establishes Theorem \ref{fluctuations of individual eigenvalues} with $m = 1$, assuming Proposition \ref{eigenvalue counting}.
\begin{lemma} Proposition \ref{eigenvalue counting} and Theorem \ref{fluctuations of individual eigenvalues} are 
equivalent. 
\label{eigenvalue counting to individual fluctuations}
\end{lemma}
\begin{proof}
We discuss the real case, the complex case being identical. We use the notation
	$$
		X_k = \frac{\lambda_k - \gamma_k}{ \sqrt{ \frac{4\log N}{\left(4-\gamma_k^2\right)N^2}  }  },
		\quad
		Y_k(\xi) = \left| \left\{ j\,:\, \lambda_j \leq \gamma_k +
			\xi \sqrt{ \frac{4\log N}{\left(4-\gamma_k^2\right)N^2}  } \right\} \right|,
	$$
with $X_k$ as in (\ref{def: X_i}). Let
	\[
		e\left(Y_k(\xi)\right) = 
		N\int_{-2}^{\gamma_k + \xi\sqrt{ \frac{4\log N}{\left(4-\gamma_k^2\right)N^2}  }} \rho_{sc}(x) \, \rd x,
		\quad
		v\left( Y_k(\xi) \right) = \frac{1}{\pi}\sqrt{\log N }.
	\]
The main observation is that 
	\begin{align*}
		\P \left( X_k < \xi\right) = 
		\P \left( Y_k(\xi) \geq k \right) = \P \left( \frac{Y_k(\xi) - e\left(Y_k(\xi)\right)}{v\left( Y_k(\xi\right)} \geq 
			\frac{k - e\left(Y_k(\xi)\right)}{v\left( Y_k(\xi\right)} \right).
	\end{align*}
	Now observe that by (\ref{quantiles}),
	\[
		N\int_{-2}^{\gamma_k + \xi \sqrt{ \frac{4\log N}{\left(4-\gamma_k^2\right)N^2}  }} \rho_{sc}(x) \, \rd x = k + \frac{\xi}{\pi}\sqrt{\log N} + \oo\left( 1 \right).
	\]
	This proves the claimed equivalence.
\end{proof}
\noindent The proof of Proposition \ref{eigenvalue counting} closely follows the proof of Theorem \ref{main theorem}.
In particular, the proof proceeeds by comparison with GOE and GUE. In the following, we first state what is known in the GOE
and GUE cases. Then we indicate the modifications to the proof of Theorem \ref{main theorem} required to establish
Proposition \ref{eigenvalue counting}. 

\subsection{The GOE and GUE cases. }
Gustavsson \cite{Gus2005} first established the following central limit theorem in the GUE case, and O'Rourke \cite{ORo2010} established
the GOE case. Here the notation $k(N) \sim N^\theta$ is as in (\ref{gustavsson sim}).
\begin{theorem}[Theorem 1.3 in \cite{Gus2005}, Theorem 5 in \cite{ORo2010}]
Let $\lambda_1 < \lambda_2 < \dots < \lambda_N$ be the eigenvalues of a GOE (GUE) matrix.
Consider $\left\{ \lambda_{k_i} \right\}_{i=1}^m$ such that $0 < k_{i+1} - k_i \sim N^{\theta_i}$ \nc,
$0 < \theta_i \leq 1$, and $k_i/N \to a_i \in (0,1)$ as $N \to \infty$. With $\gamma_k$ as in (\ref{quantiles}),
let
	\[
		X_i = \frac{\lambda_{k_i} - \gamma_{k_i}}{\sqrt{ \frac{4\log N}{\beta \left(4 - \gamma_{k_i}^2 \right) N^2}  }},
		\quad i = 1, \dots, m,
	\] 
where $\beta = 1,2$ corresponds to the GOE, GUE cases respectively. Then as $N \to \infty$,
	\[
		\P\left\{ X_1 \leq \xi_1, \dots, X_m \leq \xi_m \right\} \to \Phi_{\Lambda}\left(\xi_1, \dots, \xi_m \right),
	\]
where $\Phi_\Lambda$ is the cumulative distribution function for the $m$-dimensional normal distribution with
covariance matrix $\Lambda_{i,j} = 1 - \max\left\{ \theta_k: i \leq k < j < m \right\}$ if $i < j$, and $\Lambda_{i,i} = 1$.
\end{theorem}

\noindent By Lemma \ref{eigenvalue counting to individual fluctuations}, the real (complex) case in Proposition \ref{eigenvalue counting}
holds for the GOE (GUE) case. Therefore we can prove Proposition \ref{eigenvalue counting} by comparison, 
presenting only what differs from the proof of Theorem \ref{main theorem}. We only consider the real case, the proof in the complex case being similar. Each step below corresponds to a section in our proof of Theorem  \ref{main theorem}.

\subsection{Step 1: Initial Regularization. }
\begin{proposition}\label{smoothing im log}
		Let 
		$y_1 < y_2 < \dots < y_N$ denote the eigenvalues of a Wigner matrix satisfying (\ref{subgaussian}). Set
		\begin{equation*}
		g(\eta) =\Im \sum_{k} \left(\log \left(y_k+\ii\eta\right) - \log y_k \right)-\int_0^{\eta}N\Re \left(m_{sc}(\ii s)\right)\rd s, 
		\end{equation*}
		and recall
		$ \eta_0 =\frac{e^{ \left(\log N\right)^{ \frac{1}{4}}}}{N}$.
		Then 
		$g\left(\eta_0\right)$ converges to 0 in probability as $N\to\infty$.
\end{proposition}
\begin{proof}
Again, we choose $
		\tilde{\eta} = \frac{c_N}{N}= \frac{ \left(\log N\right)^{\frac{1}{4}} }{N}.
	$
Then
	\begin{align*}
		\E \left| g\left(\eta_0\right) - g\left(\tilde{\eta}\right) \right| 
		\leq \E \int_{\tilde{\eta}}^{\eta_0} N\left| \Re\left( s(\ii u) \right) - \Re\left(m_{sc}\left(\ii u \right) \right)\right| \, \rd u.
	\end{align*}
 Theorem \ref{schlein} holds whether we consider $s$ or $\Im \left(s\right)$, so that exactly the same argument as previously shows
$\E \left| g\left(\eta_0\right) - g\left(\tilde{\eta}\right) \right| = \oo\left( \sqrt{\log N} \right)$. 

\noindent Next define $b_N = \frac{e^{-\left( \log N \right)^\frac{1}{8}}}{N}$. 
As $b_N$ is below the microscopic scale, by Corollary \ref{micro fixed energy}, $$\sum_{\left| x_k \right| \leq b_N} \left( \Im \log \left( x_k + \ii \tilde{\eta} \right) - \Im \log \left( x_k \right) \right)$$ converges to 0 in probability, as the probability it is an empty sum converges to 1.\\

\noindent Consider now
\begin{equation}\label{2ndterm}
		\sum_{\left| x_k \right| > b_N} \left( \Im \log \left( x_k + \ii \tilde{\eta} \right) - \Im \log \left( x_k \right) \right).
\end{equation}
Let $N_1(u) = \left| \left\{ x_k \leq  u \right\} \right|$
and note that
	\[
		\Im \log \left( x \right) - \Im \log \left( x + \ii \tilde{\eta} \right) = 
		\int_0^{\tilde{\eta}} \Re\left( \frac{1}{x - \ii u} \right) \, \rd u
		= \arctan\left( \frac{\tilde{\eta}}{x} \right).
	\]
To prove (\ref{2ndterm}) is negligible, it is therefore enough to bound $\E(|X|)$ where 
	\begin{align*}
		X = \int_{b_N \leq |x| \leq 10} \arctan\left( \frac{\tilde{\eta}}{x} \right) \, \rd N_1(x)
		    &= \int_{b_N}^{10} \arctan\left( \frac{\tilde{\eta}}{x} \right) \, \rd (N_1(x)+N_1(-x)-2N_1(0)).
	\end{align*}
After integration  by parts,  the boundary terms are $\oo(1)$ and 
	$$
	\tilde{\eta}\int_{b_N}^{10}\frac{\E(|N_1(x)+N_1(-x)-2N_1(0)|)}{x^2+\tilde{\eta}^2} \, \rd x
	$$
	remains.
Split the above integral into integrals over $[b_N,a]$ and $[a,10]$ where $a=\exp(C(\log \log N)^2)/N$ for a large enough $C$.
On the first domain, Corollary \ref{repulsion}  gives the bound $\E(|N_1(x)+N_1(-x)-2N_1(0)|)\leq C Nx+\delta$ for any small $\delta>0$.
On the second domain,  by rigidity \cite{ErdYauYin2012Rig} we have $|N_1(x)+N_1(-x)-2N_1(0)|\leq \exp(C(\log \log N)^2)$, so that the contribution 
from this term is also $\oo\left( \sqrt{\log N} \right)$. 
\end{proof}

\subsection{Step 2: Coupling of Determinants. }
With the notation of Section \ref{DBM} we have,
	\[
		e^{t/2}\Im\left( f_t\left( \ii \eta_0 \right) \right) = 
		\frac{\rd}{\rd \nu} \sum_{k=1}^N \left( \Im \log\left( \lambda_k^{(\nu)}(t) + \ii \eta_0 \right) \right).
	\]
We can therefore proceed in the same way as Proposition \ref{prop:advection} to prove the following.
	\begin{proposition}
	\label{prop:advection im}
	Let $\epsilon >0$, $\tau = N^{-\epsilon}$ and let $z_\tau$ be as in (\ref{def:ztau}) with $z = {\rm i}\eta_0$.
	Let 
		\[
			g(t,\eta) = \sum_k \left(\Im \log \left( x_k(t) + {\rm i}\eta \right) - \Im \log \left(y_k(t) + {\rm i}\eta\right)  \right)
		\]
	Then for any $\delta > 0$, 
		$ \lim_{N\to\infty} \mathbb{P}\left(\left|
		g\left(\tau, \eta_0\right) - g\left(0, z_\tau\right)\right| > \delta  \right) = 0.$
	\end{proposition}
	
\subsection{Step 3: Conclusion of the Proof. } We reproduce the reasonning from (\ref{eqn:1}) to (\ref{eqn:2}) to prove Proposition \ref{eigenvalue counting} in the real symmetric case.
 From \cite{ORo2010} and Proposition \ref{smoothing im log}, for some explicit deterministic  $c_N$ we have

\begin{equation}\label{eqn:11}
\frac{ 			\sum_{k =1}^N \im\log \left(x_k(\tau) + {\rm i} \eta_0\right)+c_N}{\sqrt{\log N}}\to\mathscr{N}(0,1),
\end{equation}
and Proposition \ref{prop:advection im} implies that
$$
\frac{ 			\sum_{k =1}^N \im\log \left(y_k(\tau) + {\rm i} \eta_0\right)+c_N}{\sqrt{\log N}}
+\frac{ 			\sum_{k =1}^N \im\log \left(x_k(0) + z_\tau\right)-\sum_{k =1}^N \im\log \left(y_k(0) + z_\tau\right)}{\sqrt{\log N}}
\to\mathscr{N}(0,1).
$$

	\noindent Lemmas \ref{expectation im log} and  \ref{variance im log} show that the second term above, call it $X$, satisfies $\E(X^2)<C\e$, for some universal $C$. Thus for any fixed smooth and compactly supported function $F$,
	\begin{align*}
	\E \left(F\left(\frac{ 			\sum_{k =1}^N \im\log \left(y_k(\tau) + {\rm i} \eta_0\right)+c_N}{\sqrt{\log N}}\right)\right)&=\E \left(F\left(\frac{ 			\sum_{k =1}^N \im\log \left(x_k(\tau) + {\rm i} \eta_0\right)|+c_N}{\sqrt{\log N}}+X\right)\right)
	+\OO\left(\|F\|_{\rm Lip}(\E\left(X^2\right))^{1/2}\right)\\
	&=\E \left(F(\mathscr{N}(0,1))\right)+\oo(1)+\OO\left(\e^{1/2}\right).
	\end{align*}
With Theorem \ref{4 moment matching theorem} (its proof applies equally to the imaginary part), the above equation implies
$$
\E\left( F\left(\frac{\im\log\det(W+\ii\eta_0)+c_N}{\sqrt{\log N}}\right)\right)=\E \left(F(\mathscr{N}(0,1))\right)+\oo(1)+\OO\left(\e^{1/2}\right),
$$		
and by Proposition \ref{smoothing im log}, we obtain
\begin{equation}\label{eqn:12}
\E \left(F\left(\im \frac{\log \det W+\frac{N}{2}}{\sqrt{\log N}}\right)\right)=\E \left(F(\mathscr{N}(0,1))\right)+\oo(1)+
\OO\left(\e^{1/2}\right).
\end{equation}
Since $\e$ is arbitrarily small, this concludes the proof.

\begin{lemma}
	\label{expectation im log}
	Recall the notation $\tau = N^{-\epsilon}$ and let $\{x_k\}_{k=1}^N$, $\{y_k\}_{k=1}^N$ denote the
	eigenvalues of two Wigner matrices, $W_1$ and $W_2$. Then
	\[
			\lim_{N\to\infty}	\E \left( \sum_{k=1}^N \Im \log \left(x_k + \ii \tau\right)
				- \sum_{k=1}^N \Im \log \left(y_k + \ii \tau\right) \right) = \OO(1).
	\]
\end{lemma}
\noindent The proof of this lemma requires only trivial adjustments of the proof of Lemma \ref{expectation}, details are left to the reader.
Finally, we also have the following bound on the variance.

\begin{lemma}
	\label{variance im log} Recall the notation $\tau = N^{-\epsilon}$ and let $\{x_k\}_{k=1}^N$, denote the
	eigenvalues of a Wigner matrix $W$. Then there exists $\e_0>0$ such that for any $0<\e<\e_0$ we have
	\begin{equation}
		\label{im log variance}
			\var\left( \sum_{k=1}^N\Im \log\left( x_k + \ii \tau \right) \right) \leq C(1+ \e \log N ).
		\end{equation}
\end{lemma}

	\noindent For the proof, let  $\chi_{[-5,5]}$ is a smooth indicator of the interval $[-5,5]$ and
	$
		\varphi_N(x) = \chi(x)\Im \log\left( x + \ii \tau \right).
	$
Our first proof of Proposition \ref{variance} shows it is enough to check that
	$
		\int \left| \hat{\varphi}_N(\xi) \right|^2 |\xi| \, \rd \xi = \OO\left(1+ \log \tau \right).
	$
We can verify this bound by integrating by parts as before. Alternatively, we can use the second proof of  Proposition \ref{variance}  based 
on the resolvent, which applies without changes.

\end{document}